\documentclass[12pt]{amsart}
\usepackage{amsmath}
\usepackage{amssymb,amsthm, lipsum}
\usepackage{mathtools}
\usepackage[utf8]{inputenc}
\usepackage{xcolor}
\usepackage[most]{tcolorbox}
\definecolor{block-gray}{gray}{0.85}
\newtcolorbox{blockquote}{colback=block-gray,grow to right by=-1mm,grow to left by=-1mm,boxrule=0pt,boxsep=0pt,breakable}

\usepackage[margin=2.5cm]{geometry}
\usepackage{graphicx}
\usepackage{hyperref}
\theoremstyle{definition}

\newtheorem{theorem}{Theorem}

\newtheorem{lemma}[theorem]{Lemma}
\newtheorem{prop}[theorem]{Proposition}
\newtheorem{corollary}[theorem]{Corollary}

\newtheorem{Question}[theorem]{Question}
\theoremstyle{definition}
\newtheorem{definition}[theorem]{Definition}
\newtheorem{remark}[theorem]{Remark}

\newtheorem{fact}[theorem]{Fact}

\newcommand{\ep}{\varepsilon}

\newcommand{\res}{\restriction}

\newcommand{\sm}{\setminus}

\newcommand{\ww}{\omega^\omega}

 \newcommand{\sU}{\mathcal{U}}

\newcommand{\sS}{\mathcal{S}}
\newcommand{\sF}{\mathcal{F}}

\newcommand{\zfc}{\ensuremath {\mathsf{ZFC}}}
\newcommand{\zf}{\ensuremath {\mathsf{ZF}}}
\newcommand{\ch}{\ensuremath {\mathsf{CH}}}
\newcommand{\ac}{\ensuremath {\mathsf{AC}}}
\newcommand{\dc}{\ensuremath {\mathsf{DC}}}

\newcommand{\norm}[1]{\Big\lVert#1\Big\rVert}

\def\N{{\mathbb N}} \def\R{{\mathbb R}} \def\Z{{\mathbb Z}} \def\K{{\mathbb K}}

  \def\cT{{\mathcal T}}

\DeclareMathOperator\supp{supp}

\newcommand{\jt}{\text{JT}} 
\newcommand{\J}{\text{J}}

\title{On coarse geometry of separable dual Banach spaces}
\author{Stephen Jackson}
\author{Cory Krause}
\author{B\"unyamin Sar\i}
\address{Stephen Jackson and B\"unyamin Sar\i, Department of Mathematics, University of North Texas, 1155
Union Circle \#311430,
Denton, TX 76203-5017}
\email{Stephen.Jackson@unt.edu, bunyamin@unt.edu}
\address{Cory Krause, Department of Mathematics, LeTourneau University, 2100 South Mobberly, Longview, TX 75602}
\email{corykrause@letu.edu}
\thanks{The first author was supported by NSF grant DMS-1800323.}


\begin{document}

\begin{abstract}
We study the obstructions to coarse universality in separable dual Banach spaces. We prove coarse non-universality of several classes of dual spaces, including those with conditional spreading bases, as well as generalized James and James tree spaces. We also give quantitative counterparts of some of the results, clarifying the distinction between coarse non-universality and the non-equi-coarse embeddings of the Kalton graphs. Unique to our approach is the use of a Ramsey ultrafilter. While the existence of such ultrafilters typically requires \ch, we are able to show that the conclusions of our theorems follow from \zfc\, alone via an absoluteness argument. Finally, we also show how our techniques can be used to prove various previously known results in the literature.

\end{abstract}
\maketitle

\tableofcontents

\section{Introduction}

We study the obstructions to coarse universality in separable {\em dual} Banach
spaces. That is, we seek conditions on such spaces which do not admit a coarse
embedding of $c_0$ (see 2.1 below for the definition of a coarse embedding). Recall that $c_0$ is universal for all separable metric spaces in the Lipschitz (and in particular, coarse) category. {\em In full generality, it remains open whether $c_0$ coarsely embeds into any
separable dual Banach space.} (See \cite{BLPP} for a discussion and a partial result for coarse-Lipschitz embeddings.) 

The notion of coarse embedding arises naturally in geometric group theory and was first introduced (under a different name) and studied by Gromov \cite{Gro1,Gro2}. As the definition in the next section makes clear, coarse embedding is a very weak notion of embedding. In fact, as Gromov pointed it out, it was not at all obvious whether there exist {\em any} obstructions to embedding an arbitrary separable metric space into an infinite-dimensional Hilbert space or another `nice' Banach space. As it turned out there are such spaces.  While it wasn't the first example, Johnson and Randrianarivony \cite{JR}, for instance, proved that for $p>2$, $\ell_p$ does not coarsely embed into a Hilbert space.

On the other hand, Yu's seminal work \cite{Yu} showed that the existence of such embeddings has deep consequences in topology and geometric group theory, particularly for the Novikov conjecture and the coarse Baum-Connes conjecture. These are often viewed as infinitesimal analogues of the Borel conjecture, which asserts that two aspherical closed manifolds are homeomorphic whenever their fundamental groups are isomorphic.

Yu's results initiated an extensive line of research, across several areas of mathematics including Banach space theory aimed at identifying obstructions to coarse embeddings of metric spaces into certain nice classes of Banach spaces. Providing a comprehensive survey of these developments is beyond the scope of this introduction. We refer the reader to the excellent expositions in the introductions of \cite{EMN,BLS} and the references therein for further background.

Kalton studied coarse embeddings in Banach spaces in a milestone work \cite{K} where he showed that $c_0$ (and certain other Banach spaces) do not coarsely embed into reflexive Banach spaces. Kalton notably introduced an invariant, nowadays called the Kalton's property $\mathcal Q$ and proved a concentration inequality for the Kalton's interlacing graphs, which are used as (quantitative) test spaces for obstructions to the coarse universality. These ideas have had significant influence on subsequent developments in the coarse geometry of Banach spaces (cf. \cite{BLMS, BLS, LPP}), including the present work.

Our work is centered around a beautiful argument of
Kalton \cite{K} which we refer to as {\em the interlacing argument}. The core of
the argument is very simple to describe. Suppose $\phi:c_0\to X$ is a coarse
embedding. Consider the summing basis $(s_i)$ of $c_0$ and for every $k$-tuple
$\vec n=(n_1, \ldots, n_k)\in \N^k$, consider the vector $\phi\big(\sum_{i=1}^k s_{n_i}\big)$ which we abbreviate as $\phi(\vec n)$. By the fact that $\phi$ is a coarse embedding, $\|\phi(\vec n)\|\to \infty$ as $k\to\infty$ since $\|\sum_{i=1}^k s_{n_i}\|=k$. On
the other hand, since $\phi$ is $K$-Lipschitz for large distances for some $K$,
for interlacing tuples $n_1<m_1<\ldots<n_k<m_k$ we have $\|\phi(\vec n)-\phi(\vec m)\|\le K$ since $\|\sum_{i=1}^k s_{n_i}-\sum_{i=1}^k
s_{m_i}\|=1$ in $c_0$. Thus, if for infinitely many $k$'s we can find interlacing
tuples so that the inequality
$$(*)\ \|\phi(\vec n)\|\le C\|\phi(\vec n)-\phi(\vec m)\|$$ holds for some interlacing tuples $\vec{m},\vec{n}\in M^k$ ($C$ could be replaced by any function $C(k)$ which grows slower
than the coarse expansion map of $\phi$), then we would arrive at a contradiction, thus such a
coarse embedding is not possible. 

In order to show that $(*)$ holds for a space one needs to understand the
possible form of the vectors $\phi(\vec n)$ and, in some cases,
their norming functionals $\phi^*(\vec n)$, and use appropriate assumptions
on $X$ and its dual $X^*$. The heart of the problem is combinatorial and it
depends on whether $X$ and $X^*$ allow certain Ramsey type stabilizations. It
turns out that for certain classes of spaces such as {\em reflexive spaces} or {\em dual spaces
with asymptotic unconditional structure} this is rather easy to do. These spaces satisfy {\em Kalton's property $\mathcal Q$}, that implies that there exists an infinite $M\subseteq \N$ so that $(*)$ holds {\em for all} tuples $\vec n\in M^k$. In the literature, such a result is often called a \emph{concentration inequality}. Of course, the property $\mathcal Q$ is not a necessary obstruction to the coarse universality. The property $\mathcal Q$ covers limited instances of Kalton's interlacing scheme. As we will show, variations of Kalton's scheme can be made to work for more general classes of separable dual spaces. In this way, we are able to show coarse non-universality in instances where concentration inequalities are not possible. 

The starting point for the proofs of most of our theorems is a general {\em asymptotic linearization}
theorem, Theorem \ref{linearization}. First, we may assume for the purpose of the problem that $X$
has a boundedly complete basis since every separable dual Banach space embeds
into a space with a boundedly complete basis \cite{DFJP}. Then, given a nonprincipal filter $\sU$, Theorem \ref{linearization} asserts that, for a $\sU$
positive set of tuples, $\phi(\vec n)$ is close to being linear in the
following sense. For all $\ep>0$,
there is a $\sU$ positive set of tuples (in the iterated sense defined in the statement of Theorem \ref{linearization}) so that
$\phi(\vec n)$ is $\ep$-close to the sum of successive block vectors
$h_0, h(n_1), \ldots, h(n_1,\ldots, n_k)$ with respect to the basis. We also have a similar linearization for norming
functionals for $\phi(\vec n)$'s in reflexive spaces.

While the form of the argument used to prove Theorem \ref{linearization} is well-known to experts, we believe it is worthwhile to present it in this general setting where it can be applied to many situations. Theorem \ref{linearization} is stated in terms of a nonprincipal filter, but if we take $\sU$ simply to be the Frechet filter of cofinite sets, we obtain Corollary \ref{lin-asy}. Here the supports of the vectors
$h(n_1, \ldots, n_i)$'s may be taken arbitrarily far out when the $n_i$'s are. In the language
of asymptotic structure of \cite{MMT}, $(h(n_1, \ldots, n_i))_{i=1}^k$ spans an asymptotic
space for $X$. In Section \ref{applications}, we use this corrollary to give unified proofs
of several known results in the literature (e.g., main results of \cite{K},
\cite{BLS}, and \cite{BLMS}).

For instance, we deduce that {\em $c_0$ does not
embed into a reflexive space or a separable dual space with asymptotic unconditional structure}
(Theorem \ref{ref-asy-unc}.) We also deduce that {\em
non-reflexive spaces which admit no $\ell_1$ spreading models do not coarsely
embed into reflexive Banach spaces} (Corollary \ref{non-ref-no-l_1}). This
reproves Kalton's results \cite{K} that $c_0$ and certain non-reflexive spaces
like the James space and those with alternating Banach-Saks property do not
coarsely embed into reflexive spaces. Moreover, Theorem \ref{linearization} also asserts that $\|h(n_1,\ldots, n_i)\|\le K$,
$1\le i\le k$. This in turn can be used to prove the main result of \cite{BLMS} that {\em any space that coarsely embeds
into a reflexive asymptotic-$c_0$ space must be a reflexive asymptotic-$c_0$
space} (Corollary \ref{asy-c_0}).

Finding obstructions to the coarse non-universality of non-reflexive spaces is
more involved. If we take $\sU$ in Theorem $\ref{linearization}$ to be a Ramsey ultrafilter, we obtain Corollary \ref{lin-Ramsey} where there exists a $\sU$ large set $H$ such that the linear approximations of Theorem $\ref{linearization}$ hold for all vectors $\phi(\vec n)$ when the tuple $\vec n$ is taken from $H$. It is from this corollary that we derive the main results of the paper. Specifically, we study those spaces whose prototype examples are the James space and the James tree space.  The James space has two important properties
for the purposes of the paper, each of which is an obstruction to coarse
universality. One is that it has a {\em boundedly complete conditional spreading
basis.} We prove that spaces with such bases are not coarsely universal (Theorem
\ref{spreading}). The second is that it is quasi-reflexive. More generally, we
study the spaces whose dual is of the form $X^*=[e^*_i]\oplus Z^*$ where
$e^*_i$'s are the biorthogonal functionals to the boundedly complete basis
$(e_i)$. Theorem \ref{linearization} together with an interlacing argument
imply that if $X$ is coarsely universal then, roughly speaking, for a large set
of tuples the norming functionals $\phi^*(\vec n)$ must belong to
$Z^*$. It follows via a bit of combinatorics and again an  interlacing argument that
$X$ is not coarsely universal if $Z^*$ is finite dimensional (Theorem
\ref{quasireflexive}). {\em Both of these arguments use an infinite pigeonhole principle via Ramsey ultrafilters and are not quantitative. Although the existence of Ramsey ultrafilters typically requires $\ch$, our results do not by an absoluteness principle (see Section \ref{ch} for details).}  In the case of James space $\J$ via a simple combinatorial argument we get a quantitative non-embedding result that the Kalton interlacing graphs $(\N^k, d_{\K})$ do not equi-coarsely embed into $\J$ (Theorem \ref{J}). This was first proved in \cite{LPP} with a more involved argument. 
Moreover, {\em the generalized James spaces with boundedly complete bases are not coarsely universal} (Theorem \ref{J(e_i)-c_0}).  

One of our main contributions is a reduction argument for generalized James tree type spaces in which a coarse embedding is shown to be essentially supported on a finite set of branches, which is of independent interest (Section \ref{section-reduction}). As a consequence for the James tree space $\jt$ we show {\em not only that $c_0$ does not coarsely embed into $\jt$ but the Kalton's interlacing graphs $(\N, d_{\K})$ do not equi-coarsely embed into it} (Corollary \ref{JT}). Moreover, this result holds for a more general class of James tree spaces $\jt(e_i)$, hence, also for generalized James spaces $\J(e_i)$ (Corollary \ref{gen-james-graphs}), built over spaces whose asymptotic structures do not contain $\ell_{\infty}^n$'s (Theorem \ref{tree-space}).

\section{Preliminaries}
\subsection{Coarse embeddings}

Let $(X,d_X)$ and $(Y,d_Y)$ be metric spaces. A map $\phi:X\to Y$ is called a
{\em coarse embedding} whenever there exist nondecreasing functions
$\rho,\omega:[0,\infty)\to[0,\infty)$ such that
\begin{enumerate}
\item
$\rho(d_X(x_1,x_2))\leq d_Y(\phi(x_1),\phi(x_2))\leq \omega(d_X(x_1,x_2))$ for
all $x_1,x_2\in X$.
\item
$\lim_{t\to\infty} \rho(t)=\infty$.
\end{enumerate}

A family $(\phi_n)_n$ of maps is an {\em equi-coarse embedding} of a family of metric spaces $(X_n, d_{X_n})_n$ into $Y$ if there exists $\rho$ and $\omega$ as above so that the above conditions are satisfied for all $\phi_n:X_n\to Y$ simultaneously.

Coarse maps send bounded sets in $X$ to bounded sets in $Y$ and unbounded sets in
$X$ to unbounded sets in $Y$, but they do so in a uniform way within fixed
bounds determined by the functions $\rho$ and $\omega$. Notice that it is sufficient to define $\phi$ on a dense subset of $X$. What is more, it is sufficient to define $\phi$ on an $\ep$-net of $X$ (a subset $N$ of $X$ such that $\forall x\in X \exists y\in N \, d(x,y)<\ep$). If $X$ and $Y$ have vector space structure, we may assume that the map $\phi$ satisfies $\phi(0)=0$. This is because if $\phi$ is a coarse embedding, then so also is $\phi^*(x)=\phi(x)-\phi(0)$.

Recall that a metric space $X$ is \emph{metrically convex} if whenever $x_1,x_2\in
X$ and $\lambda\in(0,1)$ there exists a $z_\lambda\in X$ such that
\[
d(x_1,z_\lambda)=\lambda d(x_1,x_2) \quad  \text{ and } \quad d(z_\lambda,x_2)=(1-\lambda)d(x_1,x_2).
\]
Furthermore, a map between metric space $\phi:X\to Y$ is called
\emph{Lipschitz at large distances} provided there exist $r,K\in(0,\infty)$ such
that $d_Y(\phi(x_1),\phi(x_2))\leq K d_X(x_1,x_2)$ whenever $d_X(x_1,x_2)\geq
r$. A well-known and easy to prove fact is that if $X$ is metrically convex and $\phi:X\to Y$ is a coarse embedding, then $\phi$ is 
Lipschitz at large distances. 

In this paper, essentially all the maps $\phi$ we consider  are coarse embeddings of $c_0$ or equi-coarse embeddings of {\em the Kalton's interlacing graphs} $(\N^k, d_{\K})_{k=1}^{\infty}$ into a Banach space $X$. Recall that these graphs are defined by joining $\vec n, \vec m \in \N^k$ if they are {\em interlacing}, that is either $n_1\le m_1<\ldots\le n_k\le m_k$ or $m_1\le n_1<\ldots\le m_k\le n_k$, and the distance $d_{\K}(\vec n, \vec m)$ is the shortest path metric on the graph (\cite{K}). Thus, in our setting, if $\phi$ is a coarse embedding, then there exist a constant $K$ and a non-decreasing function $\rho$ with $\lim_{t\to\infty} \rho(t)=\infty$ such that for all $x_1, x_2$ with $d(x_1, x_2)\ge 1$ we have
$$\rho(d(x_1, x_2))\le \|\phi(x_1)-\phi(x_2)\|\le
K d(x_1, x_2).$$ For our purposes, the reader may take this as the definition of a coarse embedding and, similarly, the corresponding statement for equi-coarse embeddings.

\subsection{Ramsey ultrafilters and the absoluteness principle}\label{ch}
A collection $\sU$ of nonempty subsets of $\N$ is called a filter on
$\N$ provided it is closed under finite intersections and
supersets. Such a collection is called an ultrafilter if for all
$A\subseteq\N$ precisely one of $A$ or $\N\sm A$ is in $\sU$. Of
particular importance are the nonprincipal ultrafilters which consist
of only infinite sets. Unique to our presentation is the use of
so-called Ramsey ultrafilters which are nonprincipal ultrafilters
which satisfy a further additional property. When $A\subseteq \N$, let
$[A]^k$ represent the subsets of $A$ having size $k$. In the paper, we freely identify $[A]^k$ with $A^k$ which is the set of increasing $k$-tuples from $A$.

Then a
nonprincipal ultrafilter is said to be \emph{Ramsey} provided that:
$$\text{For all } k, m\in\N \text{ and each } f:[\N]^k\to m \text{
  there exists } H\in\sU \text{ so that } f \text{ is constant on }
[H]^k$$ We will sometimes refer to $f$ as a \emph{partition} of the
$k$-tuples of $\N$ into $m$ pieces and call $H\in\sU$
\emph{homogeneous} for the partition.  While the existence of Ramsey
ultrafilters is independent of $\zfc$ set theory, it is well-known
that their existence follows from the continuum hypothesis $\ch$
\cite{J}.

Recall Ramsey's theorem which states that for all $k, m\in\N$ and each
$f\colon [\N]^k\to m$ there exists an infinite $H\subseteq\N$ so that $f$ is
constant on $[H]^k$. When we prove that $c_0$ does not coarsely embed
into a space, our arguments often need the homogeneous set $H$ to be
in a fixed ultrafilter $\sU$ thus giving rise to the need for Ramsey
ultrafilters. However, whenever we are able to prove the stronger
claim that the Kalton's interlacing graphs $(\N^k,d_\K)_k$ do not
equi-coarsely embed into a space, the reader will note that our proofs
could be easily reworded in such a way as to use only Ramsey's
theorem. In the main results of our paper, proofs are given using Ramsey ultrafilters merely for
uniformity of presentation.

Because Ramsey ultrafilters are important for the arguments of this paper, and for the 
sake of completeness, we give a little more background. An ultrafilter $\sU$ on $\N$ is called
{\em selective} if for every partition $\N= \bigcup_{n \in \N} A_n$ of $\N$ into a countably infinite 
number of pieces each of which is not in $\sU$, there is a set $A\in \sU$ with $|A\cap A_n|\leq 1$
for all $n$ (equivalently, $|A\cap A_n|=1$ for all $n$). That is, we may select a point from each 
$A_n$ to get a set in the ultrafilter. A weaker notion is that of the ultrafilter $\sU$ being a $P$-point,
in which ``$|A\cap A_n|\leq 1$'' in the definition is replaced with ``$|A\cap A_n|$ is finite.'' 
Under $\ch$, there are selective ultrafilters. In Theorem~7.8 of \cite{J} a proof of the existence
of selective ultrafilters assuming $\ch$ is given (where there they are called ``Ramsey ultrafilters''). 
An ultrafilter is Ramsey if and only if it is selective. 
For a proof of this equivalence and  with various other definitions of Ramsey ultrafilter, see \cite{B}.
For convenience, we give here a short direct proof that $\ch$ implies the existence of Ramsey ultrafilters.

\begin{fact}
Assume $\zfc+\ch$. Then Ramsey ultrafilters exist. 
\end{fact}

\begin{proof}
From $\ac$ and $\ch$ we may enumerate the partitions of $\N^k$ as $\{P_\alpha\}_{\alpha<\omega_1}$. 
Here, each $P_\alpha \colon \N^k \to \ell$ where $k=k(\alpha)$, $\ell=\ell(\alpha)$ depend on $\alpha$. 
We define infinite sets $X_\alpha \subseteq \N$ inductively and maintain that $X_{\alpha}\subseteq^* X_\beta$ 
for all $\beta <\alpha$ where 
$A\subseteq^* B$ means $A\setminus B$ is finite. Let $X_0=\N$. For $\alpha$ limit,
we can easily get $X_\alpha$ so that for all $\beta <\alpha$ we have $X_\alpha\subseteq^* X_\beta$
using the fact that $\alpha$ is countable so that $\{ X_\beta\}_{\beta<\alpha}$ can be written 
as $\{ Y_n\}_{n \in \N}$ and the fact that any finite number of the $X_\beta$ (so also the $Y_n$) have a non-empty
intersection from the induction hypothesis (given $X_{\beta_1},\dots, X_{\beta_n}$, all but finitely
many of the elements of $X_{\beta_n}$ are in $X_{\beta_0}\cap \cdots \cap X_{\beta_{n-1}}$). 
For $\alpha+1$ a successor ordinal, let $X_{\alpha+1}\subseteq X_\alpha$ be a homogeneous set 
for the partition $P_\alpha$, which is possible from Ramsey's theorem as $X_\alpha$ is infinite. 
Define $\sU$ by $A\in \sU$ iff $X_\alpha \subseteq^* A$ for some $\alpha$. $\sU$ is an ultrafilter
since for any $A\subseteq \N$ we can view $A$ as giving a partition of $\N$ (with $\ell=2$)
by the characteristic function of $A$. By construction, $\sU$ contains a homogeneous set
for every partition $P\colon \N^k \to \ell$.
\end{proof}

The proofs of our main resuts in this paper make use of Ramsey ultrafilters, which cannot be shown
to exist in $\zfc$. However, results from logic show that our main results do not require $\ch$ 
as a hypothesis. For completeness, we briefly summarize this argument.

Many of the theorems of our paper will be of the form: ``For every
separable Banach space $X$ of a certain kind and every map $\phi$
defined from $c_0$ into $X$ and every nondecreasing real-valued
functions $\rho$ and $\omega$ such that $\rho$ limits to $\infty$, one
of the inequalities in (1) above fails.''  We will compute the exact complexity
of such statements below in Lemma~\ref{comp} and state an absoluteness 
principle for an even wider class of statements in Fact~\ref{rufact} below.

A statement $\varphi(y_1,\dots,y_m)$ is said to be a {\em projective} in $x$ statement if it is of the form 
\[
\exists x_1 \forall x_2 \cdots \exists x_n\ \psi(y_1,\dots,y_m,x_1,\dots, x_n)
\]
where the quantifiers are over reals (or elements of a Polish space)
and $\psi(\vec y,\vec x)$ is a Borel in $x$ statement about $\vec y, \vec x$, that is, it is the statement that 
$(\vec y,\vec x)$ lies in the Borel set coded by $x$ (see page 504 of \cite{J} for a discussion
of coding Borel sets by reals). The statement $\varphi(\vec y)$ is a $\Sigma^1_1$ in $x$ 
statement, abbreviated $\Sigma^1_1(x)$, if it is of the form $\exists x_1 \psi(\vec y, x_1)$
where $\psi(\vec y,x_1)$ is a Borel in $x$ statement about $\vec y, x_1$. 
We say the statement is $\boldsymbol{\Sigma}^1_1$ if it is $\Sigma^1_1(x)$ for some $x$.
For further background on these ``lightface'' pointclasses of sets, we refer the reader
to \cite{M}.

In particular, the statements of our theorems will 
all be projective statements. More generally still, a statement is $\boldsymbol{\Pi}^2_1$
if it is of the form $\forall A\subseteq \R\ \chi(A)$, where $\chi(A)$ is a projective in $x$ statement 
for some $x$. 
A statement is $\boldsymbol{\Sigma}^2_1$ if it is the negation of a 
$\boldsymbol{\Pi}^2_1$ statement, or equivalently, of the form 
$\exists A\subseteq \R\ \chi(A)$, where $\chi$ is a projective statement. An
example of a  $\boldsymbol{\Sigma}^2_1$ statement is ``there exists a Ramsey 
ultrafilter on $\N$.'' 

The following is the logical principle we use. This principle was first introduced in \cite{P}. 
(See Lemma 5 and the following remark of that paper.) It shows that we
may assume in our proofs that $\ch$ holds, and in particular that a
Ramsey ultrafilter exists, but our conclusions will follow from $\zfc$
alone. In fact, $\zf$ plus a weak form of choice known as $\dc_{\mathbb{R}}$ will suffice
for the applications of choice in all our arguments.\footnote{$\dc_{\mathbb{R}}$ is the statement that, if $R$ is a binary relation on $\mathbb{R}$ such that for every $a\in\mathbb{R}$ there exists a $b\in\mathbb{R}$ where $a R b$, then there is a sequence $(x_n)\subseteq \mathbb{R}$ such that $x_n R x_{n+1}$ for all $n\in\N$.}

\begin{fact} \label{rufact}
Suppose $\theta$ is a  $\boldsymbol{\Pi}^2_1$ statement and 
$\theta$ is provable from ``$\zf+ \ch + \exists$ a wellordering of $\R$''. 
Then $\theta$ is provable from $\zf+\dc_{\mathbb{R}}$. 
\end{fact}

\begin{proof}[Proof (sketch)] 
Let $M$ be a model of $\zf+\dc_{\mathbb{R}}$. There is a forcing extension $M[G]$ of $M$ in which 
$\ch$ holds and the reals are wellordered, and furthermore, by $\dc_{\mathbb{R}}$, $M$ and $M[G]$ have the same
real numbers (one generically adds a bijection between $\omega_1^M$ and $\R^M$
using functions with countable support as conditions). By assumption, $\theta$ 
holds in $M[G]$. Since $M$ and $M[G]$ have the same reals, projective statements 
hold in $M$ iff they hold in $M[G]$. Since $M\subseteq M[G]$, every 
set of reals in $M$ is in $M[G]$ and so if a $\boldsymbol{\Pi}^2_1$ statement
holds in $M[G]$ it holds in $M$. 

\end{proof}

To illustrate how Fact~\ref{rufact} is used we give the following Lemma~\ref{comp}
which will compute the complexity of the statements we are interested in.

Let $X$ be a separable Banach space. Let $D=\{ d_i\}_{i \in \N}$ be a countable dense set in $X$
closed under addition and scalar multiplication by rationals. Let $f\colon D\to \R^{\geq 0}$
be the norm for $X$ restricted to $D$. The function $f$ can be viewed a map from $\N$ to 
$\ww$ (taking a Borel bijection between $\R$ and $\ww$), and since $(\ww)^\omega$ is in Borel bijection
(in fact homeomorphic with) $\ww$, we can identify $f$ with a ``real'' $x\in \ww$. In this way, every 
separable Banach space $X$ is coded by a real $x \in \ww$. We can thus make assertions about $X$
be referring to the real $x$ coding $X$.

The following lemma shows that for any separable Banach space $X$, 
the statement $\chi(X)$ that there is a coarse embedding of $c_0$ into $X$ is a $\Sigma^1_1$
statement about the real $x$ coding $X$ 
(that is, it is a $\Sigma^1_1(x)$ statement).

\begin{lemma} \label{comp}
Let $X$ be a separable Banach space.  Let $x \in \ww$ be a code for $X$. 
Then the statement ``there is a coarse embedding of $c_0$ into $X$'' is a 
$\Sigma^1_1(x)$ statement. 
\end{lemma}

\begin{proof}
Fix $x\in \ww$ coding $X$ as described above. The set $E$ of finitely supported vectors
with rational coordinates is dense in $c_0$. Given a coarse embedding $\varphi$ from 
$c_0$ into $X$, consider $\varphi\res E$.
The value of $\varphi$ at a point $z\in E$ can be described as a Cauchy sequence of elements of 
the dense set $D$. Thus, identifying  $E$ with $\N$ and sequences from $D$
as elements of $\ww$, we can view $\varphi\res E$ as a sequence from $\ww$ and 
thus further identify $\varphi\res E$ with a single $w \in \ww$ (since 
$(\ww)^\omega$ and $\ww$ are in Borel bijection, in fact are homeomorphic).

Conversely,
any $w \in \ww$ can be viewed as giving a map from $E$ into the sequences from $D$. 
This will induce a map $\varphi=\varphi_w$ from $E$ to $X$ iff 
for all $a \in E$, $w(a)$ is a Cauchy sequence in $X$.
This condition is a Borel condition on $w$, in fact it is a $\Pi^0_3(w,x)$ condition.
Let $\alpha(w)$ be the statement that $w$ codes a map $\varphi_w$ from $E$ to the space 
$X$ coded by $x$, that is, the above condition is satisfied by $w$. 
So, $\alpha(w)$ is a $\Pi^0_3(w,x)$ statement. 

Let $\beta(w)$ be the statement that the map $\varphi_w$ from $E$ to $X$
is a coarse embedding. The statement is easily $\Sigma^1_1(w,x)$ where we 
existentially quantify over the functions $\rho$, $\omega$ as in the definition of 
a coarse embedding. This is good enough for our purposes, though we note that $\beta(w)$
is acually a simpler Borel condition. For example, the existence of the upper-bound
function $\omega$ is equivalent to saying $\forall m \in \N\ \exists n\in \N\ 
\forall a,b \in E\ (\| a-b\|_{c_0} \leq m\rightarrow \| w(a)-w(b)\|_X\leq n)$. 
The reader can check that this defines a $\Pi^0_3(w,x)$ condition.

Finally, $\chi(X)$ is equivalent to the statement $\exists w \in \ww\ 
(\alpha(w)\wedge \beta(w))$ which is a $\Sigma^1_1(x)$ statement, where $x\in \ww$
is a code for $X$. This is because there is a coarse embedding from $c_0$ into $X$ iff
there is a coarse embedding from $c_0\res E$ into $X$.

\end{proof}

Since $\chi(X)$ is projective it is a $\Pi^2_1$ statement about $x$. 
So, from Fact~\ref{rufact} it follows that if we can prove $\chi(X)$ from 
$\zfc+\exists$ a Ramsey ultrafilter, then $\chi(X)$ is provable from $\zfc$ 
alone (in fact from $\zf+\dc_\R$). Thus, in the later proofs (for example, Theorem~\ref{spreading})
we may freely assume the existence of a Ramsey ultrafilter.

\section{Coarse non-universality of Schreier metric spaces $S_{\alpha}(\Z)$}

We start with a direct proof of the known fact that the Schreier metric spaces $(S_{\alpha}(\Z), d_{\infty}),\alpha<\omega_1$ are not coarsely universal \cite{BLMS2}. Recall that for countable ordinals $\alpha$, the Schreier families $S_{\alpha}$ of finite collections of natural numbers are defined recursively. We put $S_0$ be the collection of singletons, and define
$$S_{\alpha+1}=\Big\{\bigcup_{j=1}^n E_j: E_j\in S_{\alpha} \ \text{and}\ n\le E_1<\ldots<E_n, n\in\N\Big\} $$ 
If $\alpha$ is a limit ordinal with each $S_{\beta}$ defined for $\beta<\alpha$, then fix an increasing sequence $\alpha_n$ so that $\sup \alpha_n=\beta$ and put
$$S_{\beta}=\bigcup_{n=1}^{\infty}\{E\in S_{\alpha_n}: n\le E\}.$$
The Schreier families naturally generate well-founded trees on $\N$ by considering their backward closures. That is, the collection $T(S_\alpha)=\{F\preceq E: E\in S_\alpha\}$ is a well-founded tree on $\N$ where $\preceq$ represents the initial segment relation. Moreover, it is well-known that $o(T(S_\alpha))=\omega^\alpha+1$.

Then the Schreier metric space $(S_{\alpha}(\Z), d_{\infty})$ is the subset 
$$S_{\alpha}(\Z)=\Big\{\sum_{i\in E}a_i e_i: a_i\in \Z, E\in S_{\alpha}\Big\}$$ of $c_0$ with the metric $d_{\infty}$ induced by the sup norm $\|\cdot\|_{\infty}$, and where $(e_i)$ is the unit vector basis.

As pointed out in \cite{BLMS2}, each $S_{\alpha}(\Z)$ naturally Lipschitz embeds into the higher order Tsirelson space $T^*_{\alpha}$ and $T^*_{\alpha}$ is a reflexive Banach space. Since $c_0$ does not coarsely embed into reflexive spaces (\cite{K}), it follows that $c_0$ does not coarsely embed into $S_{\alpha}(\Z)$ either. Below we give a direct proof of this fact. This gives a good example of the general scheme of arguments in this paper, which initiated this work, and also of independent interest as it is purely combinatorial. In this case, we use only Ramsey's theorem rather than ultrafilters. Also, our proof uses only the fact that the trees generated by $S_\alpha$ are increasing and well-founded. 
Here, when we say a tree $T$ on $\N$ is {\em increasing} we mean that if $(i_0,i_1,\dots,i_n)\in T$ then 
$i_0<i_1<\cdots <i_n$.

\begin{theorem}\label{Salpha}
Let $T$ be a well-founded increasing tree on $\N$ and define the subset of $c_0$
$$T(\Z)=\Big\{\sum_{i\in E}a_i e_i: a_i\in \Z, E\in T\Big\}$$ Then there does not exist a coarse embedding from $c_0$ to $T(\Z)$. In particular, for each $\alpha<\omega_1$, the Schreier metric space $S_\alpha(\Z)$ fails to be coarsely universal.
\end{theorem}

\begin{proof}
 Since the subset of vectors taking on integer values forms a 1-net in $c_0$, it is sufficient to suppose by way of contradiction that $\phi \colon c_0(\Z)\to T(\Z)$ is a coarse embedding. Fix a constant $K>0$ so that if $x,y\in c_0(\Z)$ and $\|x-y\| \leq 1$ then $\| \phi(x)-\phi(y)\| \leq K$, where $\| \cdot \|$ refers throughout to the $c_0$ norm. Then fix $k>K$ large enough so that if $x\in c_0(\Z)$ with $\| x \|\geq k$, then $\| \phi(x)\|\geq 5K$. Henceforth, we consider $k$-tuples $\vec n$ of this fixed length $k$ and define $\phi(\vec n)$ as in the introduction. 
Since $\|\sum_{i=1}^k s_{n_i}\|=k$, this gives 
\[
5K\leq \|\phi(\vec n)\|\leq kK
\]
for all such tuples $\vec n$.

We need the following two lemmas:

\begin{lemma} \label{ramseyl}
For any $1\leq i\leq k$, $i$-tuple $(n_1,\ldots,n_i)$, and $p>n_i$, there exists a homogeneous set $H\subseteq\N$ such that  for all $\vec m \in H^k$ with $\vec m \res [1,i]= \vec n \res [1,i]$ we have $\phi(\vec n)(j)=\phi(\vec m)(j)$ for all $1\leq j\leq p$.
\end{lemma}

\begin{proof} Since $\|\phi(\vec n)\|\leq kK$, we have $|\phi(\vec n)(j)|\leq k K$ for all $\vec n$ and all $j$. Therefore, for each $1\leq j\leq p$, there are only finitely many choices for the coordinates $\phi(\vec n)(j)$, and we may partition the $(k-i)$-tuples coming after $n_i$ according to the value of these finitely many coordinates. By applying Ramsey's theorem, we obtain a homogeneous set $H$ satisfying the lemma.
\end{proof} 

\begin{lemma} \label{wfl} Let $T$ be an increasing  well-founded tree on $\N$. 
Then there does not exists an infinite increasing sequence
$j_1<j_2<\cdots$ such that for all $n$, $\{j_1,j_2,\dots,j_n\}\subseteq A_n$ for
some $A_n\in T$.
\end{lemma}

\begin{proof}
First note that for all $n\geq 1$ the first element of $A_n$ must be less than or equal to $j_1$ since $\{j_1\}\subseteq A_n$ for all $n\geq 1$. Thus, by the pigeon-hole principle,  there is some $l_1\leq j_1$ such that $(l_1)\preceq A_n$ for infinitely many $A_n$. Next note that for all $n\geq 2$ the second element of $A_n$ is less than or equal to $j_2$ since $\{j_1,j_2\}\subseteq A_n$ for all $n\geq 2$. Again, by the pigeon-hole principle, there is some $l_2\leq j_2$ such that $(l_1,l_2)\preceq A_n$ for infinitely many $A_n$. Continuing in this manner, we can obtain $l_1<l_2<\cdots$ with $l_k\leq j_k$ such that $(l_1,\ldots,l_k)\preceq A_{n_k}$ where $A_{n_k}\in T$. Since $T$ is a tree $(l_1,\ldots,l_k)\in T$ for all $k$, contradicting well-foundedness.
\end{proof}

Returning to the proof of Theorem~\ref{Salpha}, recall for any $\vec n \in \N^k$ we have $\| \phi(\vec n)\| \geq 5K$, and so for some $1\leq i\leq k$ there is a $ j \in [n_i,n_{i+1})$ with $|\phi(\vec n)(j)| \geq 5K$ (we adopt the notational convention that $n_0=0$ and $n_{k+1}=\infty$). Applying Ramsey's theorem, we obtain an $H\subseteq \N$ so that there is a fixed $i_0$ where for all $\vec n \in H^k$ there is such a $j$ in $[n_{i_0},n_{i_0+1})$. 

Partition the $(k+2)$-tuples of $H$ of the form $n_1<\cdots<n_{i_0}<p_1<p_2<n_{i_0+1}<\ldots<n_k$ according to whether or not $\|\phi(\vec n)\res [p_1,p_2)\|\geq 2K$. Suppose that we have an $H'\subseteq H$ where this property holds. Fix the first $i_0$ many elements of this $H'$ and consider a sequence $p_1^1<p_2^1<p_1^2<p_2^2<\cdots\subseteq H'$ coming afterward. By successively applying Lemma \ref{ramseyl} where $p=p_2^i$ we obtain descending sets $(H_n)$ and an increasing sequence $(j_n)$ with $j_n\in [p_1^n,p_2^n)$ such that $\{j_1,j_2\ldots,j_n\}\subseteq \operatorname{supp}(\phi(\vec{n}))$ when the first $i_0$ elements of $\vec n$ are those fixed above and the last $k-i_0$ elements of $\vec n$ come from $H_n$. This contradicts Lemma \ref{wfl}. Therefore, on the homogeneous side, we must have an $H$ where the tuples satisfy $\|\phi(\vec n)\res [p_1,p_2)\|< 2K$.

Next, partition the $(k+2)$-tuples of this $H$ of the same form according to whether or not $\|\phi(\vec n)\res [n_{i_0},p_1)\|\geq 4K$ and suppose that we have an $H'\subseteq H$ where this property holds. Then consider two $(k+2)$-tuples from $H'$ consisting of $\vec n$ with $p_1, p_2$ added and $\vec m$ with $q_1, q_2$ added such that $\vec n\res ([1,i_0-1]\cup[i_0+1,k])=\vec m\res([1,i_0-1]\cup[i_0+1,k])$ and
\[
n_{i_0}<p_1<m_{i_0}<q_1<p_2<q_2
\]
Then $\|\phi(\vec m)\res [m_{i_0},q_1)\|\geq 4K$ and $\|\phi(\vec n)\res [p_1,p_2)\|< 2K$ and since $[m_{i_0},q_1)\subseteq [p_1,p_2)$ we get $\|\phi(\vec m)-\phi(\vec n)\|> 2K$. However, $\|\sum_{i=1}^k s_{m_i}-\sum_{i=1}^k s_{n_i}\|=1$ and hence $\|\phi(\vec m)-\phi(\vec n)\|\leq K$, a contradiction. Therefore, on the homogeneous side, we must have an $H$ where the tuples satisfy $\|\phi(\vec n)\res [n_{i_0},p_1)\|< 4K$.

By an interlacing argument similar to the previous paragraph, we may further find a homogeneous $H$ so that in addition $\|\phi(\vec n)\res [p_2,n_{i_0+1})\|< 4K$. Thus we arrive at a contradiction to our choice of $i_0$, and there can be no such coarse embedding.
\end{proof}

\section{Asymptotic linearization}

The starting point for the main results of this paper is the following which we call the {\em asymptotic linearization of non-linear maps} into Banach spaces. The theorem can be stated in various settings but the general principle is the following. Suppose $\phi:(\N^k,d)\to X$ is a bounded map into a {\em separable dual Banach space} and $d$ is an arbitrary metric. Then one can stabilize the map on an asymptotic structure of $X$ so that when restricted to the asymptotic structure the map is `linear'. By linear we mean that the tuples $\vec{n}=(n_1, \ldots, n_k)$ can be chosen so that the resulting map $\phi(\vec{n})$ is a sum $h_0+\sum_{i=1}^k h(n_1, \ldots, n_i)$ of `block' vectors in $X$. In some sense the map $\phi$ asymptotically becomes `a formal identity'. Theorem \ref{linearization} is stated in terms of filters which have the advantage of generality. We use Theorem \ref{linearization} to derive Corollary \ref{lin-ultra} and \ref{lin-Ramsey} which are stated before the main proof. Afterward, we derive Corollary \ref{lin-asy} which is stated in terms of general asymptotic structure and will be important for Section \ref{applications}. Since every separable dual Banach space embeds into a space with a boundedly complete basis \cite{DFJP}, for our purposes we may assume without loss of generality that $X$ has a boundedly complete basis.


\begin{theorem}\label{linearization}
 Let $X$ be a  Banach space with a {\em
boundedly complete} basis $(e_i)$. Let $\phi:(\N^k,d)\to X$ be a bounded map
where $d$ is a metric on $\N^k$. Let $\sU$ be a non-principal filter on
$\N$.  Then for all $\ep>0$ the vectors $\phi(n_1, \ldots, n_k)$ satisfy the following:

(i) There exist $n^+_0\in\N$ and $h_0< n^+_0$ (i.e., $h_0\in [e_i]_{i< n^+_0}$)  such that 
\begin{align*}
&\text{For all $l_1>n^+_0$ for $\sU$-positive $n_1>l_1$ there exist $l_1<h(n_1)<n^+_1$}\\
&\text{\qquad\qquad\qquad\qquad for some $n^+_1>n_1$} \\
&\text{for all $l_2>n^+_1$ for $\sU$-positive $n_2>l_2$ there exist $l_2<h(n_1, n_2)<n^+_2$}\\
&\text{\qquad\qquad\qquad\qquad for some $n^+_2>n_2$}\\
&\qquad\qquad\qquad\qquad\qquad\vdots\\
&\text{for all $l_{k}>n^+_{k-1}$ for $\sU$-positive $n_k>l_{k}$ there exist}\\
&\text{$l_{k}<h(n_1, \ldots, n_k)<n^+_k$ for some $n^+_k>n_k$}
\end{align*}
 such that
 
 \begin{align}
&\norm{\phi(n_1, \ldots, n_k)-h_0-\sum_{i=1}^kh(n_1, \ldots, n_i)}< \ep.\label{permissible-sum}
\end{align}

(ii) If $\phi$ is $K$-Lipschitz then we have
\begin{align}\label{norm-bound-h_j}
\|h(n_1,\ldots, n_j)\|\le K d((n_1, \ldots, n_k),(m_1, \ldots, m_k))+\ep,\ 1\le j\le k,
\end{align}
 for some pairs satisfying $n_1=m_1, \ldots, n_{j-1}=m_{j-1}< n_j<m_j<n_{j+1}< m_{j+1}\ldots
 n_k<m_k$.

(iii) Moreover,  suppose $X^*=[e^*_i]\oplus
Z^*$ for some $Z^*\subseteq X^*$, and  $\phi^*(n_1, \ldots, n_k)$ are norming
functionals for $\phi(n_1, \ldots, n_k)$. Then there exist $h^*_0$ and tuples $(h^*(n_1, \ldots, n_i))_{i=1}^k$ whose supports with respect to $(e^*_i)$ satisfy $l_{i}<h^*(n_1, \ldots, n_i)<n^+_i$ as in (i) and $z^*_0, z^*(n_1), \ldots, z^*(n_1,\ldots, n_k)\in Z^*$ such that

\begin{equation}\label{permissible-norming}
\norm{\phi^*(n_1, \ldots, n_k)-h^*_0-\sum_{i=1}^k h^*(n_1, \ldots, n_i)-z^*(n_1, \ldots, n_k)}<\ep
\end{equation}
and for $j\le k$

$$\Big|\sum_{i=1}^{l_{j+1}}z^*(n_1,\ldots,n_k)(e_i)-\sum_{i=1}^{l_{j+1}}z^*(n_1,\ldots,n_j)(e_i)\Big|<\ep.
$$
 
\end{theorem}

\begin{remark}
We don't have an application for (\ref{permissible-norming}) above in its general form other than for reflexive spaces, $Z^*=0$, or in the case that $Z^*$ is finite dimensional.
\end{remark}

Suppose $\sU$ is an ultrafilter on $\N$ and $T\subseteq \N^{<k}$ is a tree of height $k$. 
We say $T$ is {\em $\sU$-large} if for any $(n_1,\dots,n_i)\in T$, 
$\{ n_{i+1}\colon (n_1,\dots,n_i,n_{i+1})\in T\} \in \sU$ (so, in particular
$\{ n_1\colon (n_1)\in T\} \in \sU$).

\begin{corollary}\label{lin-ultra}
If $\sU$ is an ultrafilter, then there is a $\sU$ large tree $T$ such that the statements (i), (ii), and (iii) of Theorem \ref{linearization} hold for all $(n_1, \ldots, n_k)\in T$.
\end{corollary}

If $\sU$ is a Ramsey ultrafilter, then any $\sU$-large tree $T\subseteq \N^{<k}$ contains 
a {\em $\sU$-homogeneous} tree $T' \subseteq \N^{<k}$, that is, a tree $T'$ such that 
for some $H\in \sU$, $H^k \subseteq T'$. To see this, partition the 
tuples $\vec n=(n_1,\dots,n_k)$ according to whether $\vec n \in T$. Suppose that $H$
were homogeneous for the contrary side. Using the fact that $T$ is $\sU$-large and $H\in \sU$, we can 
easily build a tuple $(n_1,\dots,n_k)$ in $T$ with $n_1,\dots,n_k\in H$, a contradiction. 

\begin{corollary}\label{lin-Ramsey}
If $\sU$ is a Ramsey ultrafilter, then there is a $H\in \sU$  such that the statements (i), (ii), and (iii) of Theorem \ref{linearization} hold for all $(n_1, \ldots, n_k)\in H^k$.
\end{corollary}

Note that for any ultrafilter $\sU$, every $\sU$-homogeneous tree is $\sU$-large. 
Later in the paper we will work with Ramsey ultrafilters, and we can then use 
Theorem~\ref{linearization} to get homogeneous 
sets $H\subseteq \N$ for which all $k$-tuples from $H$ satisfy the inequalities.

\begin{proof}
Let $k\in\N$ and $\phi:\N^k\to X$ be given as in the statement. We will show the construction of \ref{permissible-sum} and
\ref{permissible-norming} simultaneously. \ref{permissible-sum} and
\ref{norm-bound-h_j} do not use the additional assumption on $X^*$.

Consider the weak-star topology on $X$. Then the unit ball of $X$ is compact with respect to this topology. Let $\phi(n_1, \ldots, n_{k-1})$ be a cluster point of the sequence $\phi(n_1, \ldots, n_{k-1}, n_{k})$ with respect to $\sU$, which we denote by 
$$\phi(n_1, \ldots, n_{k-1})=w^*-\text{cl}_{n_k, \sU}\ \phi(n_1, \ldots, n_{k-1}, n_{k}))$$

Consider the iterated limits 

\begin{equation}\label{iterated}
\begin{aligned}
\phi(n_1, \ldots, n_{k-1})&:=w^*-\text{cl}_{n_k, \sU}\ \phi(n_1, \ldots,n_{k-1}, n_k);\\
\phi(n_1, \ldots, n_{k-2})&:=w^*-\text{cl}_{n_{k-1}, \sU}\ \phi(n_1, \ldots,n_{k-2}, n_{k-1});\\
&\vdots\\
\phi(n_1)&:=w^*-\text{cl}_{n_2, \sU}\ \phi(n_1, n_2);\\
\Phi&:=w^*-\text{cl}_{n_1, \sU}\ \phi(n_1).
\end{aligned}
\end{equation}

For any $(n_1,\ldots, n_k)\in \N^k$, let $\phi^*(n_1,\ldots, n_k)$ with
$\|\phi^*(n_1,\ldots, n_k)\|=1$ denote a norming functional for
$\phi(n_1,\ldots, n_k)$. By our assumption the functionals are of the form
$\phi^*(n_1,\ldots, n_k)=x^*(n_1,\ldots, n_k)+z^*(n_1,\ldots, n_k)$ for some
$x^*(n_1,\ldots, n_k)\in [e^*_i]$ and $z^*(n_1,\ldots, n_k)\in Z^*$.

In the following, the cluster points are chosen simultaneously in $X$ and $X^*$. That is, 
consider the product space $X\times X^*$ with the corresponding weak-star topologies. We let $(\phi(n_1, \ldots, n_{k-1}), \phi^*(n_1, \ldots, n_{k-1}))$ be a cluster point of the sequence $$(\phi(n_1, \ldots, n_{k-1}, n_{k}),\phi^*(n_1, \ldots, n_{k-1},n_k))_{n_k}$$ with respect to $\sU$.

In particular, we have

\begin{equation}\label{iterated}
\begin{aligned}
\phi^*(n_1, \ldots, n_{k-1})&=w^*-\text{cl}_{n_k, \sU}\ \phi^*(n_1, \ldots,n_{k-1}, n_k);\\
\phi^*(n_1, \ldots, n_{k-2})&=w^*-\text{cl}_{n_{k-1}, \sU}\ \phi^*(n_1, \ldots,n_{k-2}, n_{k-1});\\
&\vdots\\
\phi^*(n_1)&=w^*-\text{cl}_{n_2, \sU}\ \phi^*(n_1, n_2);\\
\Phi^*&=w^*-\text{cl}_{n_1, \sU}\ \phi^*(n_1).
\end{aligned}
\end{equation}

where these limits are of the form 
\begin{equation*}
\begin{aligned}
\phi^*(n_1, \ldots, n_{k-1})&=x^*(n_1, \ldots, n_{k-1})+z^*(n_1, \ldots, n_{k-1})\\
 \vdots\\
\phi^*(n_1)&=x^*(n_1)+z^*(n_1)\\
\Phi^*&=x^*_0+z^*_0
\end{aligned}
\end{equation*}
for some $x^*(n_1, \ldots, n_{k-1}), \ldots, x^*(n_1), x^*_0\in [e^*_i]$, and
 $z^*(n_1, \ldots, n_{k-1}), \ldots, z^*(n_1), z^*_0\in Z^*$.

Let $\ep>0.$ Let $n^+_0$ be such that

\begin{equation*}
\|\Phi-P_{n^+_0}(\Phi)\|<\frac{\ep}{3k}\ \text{and}\ \|x^*_0-P^*_{n^+_0}(x^*_0)\|<\frac{\ep}{3k}
\end{equation*} 
where $P, P^*$ are basis projections,
and put
\begin{equation*}
h_0:=P_{n^+_0}(\Phi)\ \text{and}\ h^*_0:=P^*_{n^+_0}(x^*_0).
\end{equation*} 

Let $l_1>n^+_0$. Note that 
$$H_1:=\{n_1\in\N: |(\Phi-\phi(n_1))(e^*_i)|,
|(x_0^*-x^*(n_1))(e_i)|, |(z_0^*-z^*(n_1))(e_i)|<\frac{\ep}{3kl_1}, i\le l_1\}$$
is a $\sU$-positive set.
Thus for all $n_1\in H_1$ we have

\begin{equation}\label{m_1}
\begin{aligned}
\|P_{(0,l_1]}[\Phi-\phi(n_1)]\|<\frac{\ep}{3k}
,\ \|P^*_{(0,l_1]}[x_{0}^*-x^*(n_1)]\|<\frac{\ep}{3k},\ \sum_{i=1}^{l_1}|(z_0^*-z^*(n_1)(e_i)|<\frac{\ep}{3k}
\end{aligned} 
\end{equation}

Let $n_1\in H_1$ and $n^+_1> n_1$ be such that 

\begin{align}\label{n^+_1}
\|\phi( n_1)-P_{n^+_1}\phi( n_1)\|<\frac{\ep}{3k}\  \text{and}\ \|x^*( n_1)-P^*_{n^+_1}x^*( n_1)\|<\frac{\ep}{3k}. 
\end{align}
Put 
\begin{align}
h(n_1):=P_{(l_1, n^+_1]}\phi( n_1)\ \ &\text{and}\ \ h^*(n_1):=P^*_{(l_1, n^+_1]}x^*( n_1)
\end{align}

Thus we have constructed blocks $h_0, h^*_0$ and, for $\sU$-positive $n_1>l_1$,
blocks $h(n_1), h^*(n_1)$ whose supports satisfy $h_0, h^*_0<n^+_0<l_1<h(n_1),
h^*(n_1)<n^+_1$.  From the construction we have the estimates

\begin{align*}
&\|\phi( n_1)-h_0-h(n_1)\|<\ep/k,\\
&\|\phi^*(n_1)-h_0^*-h^*(n_1)-P^*_{(0,l_1]}z_0^*-P^*_{(l_1,\infty)}z^*(n_1)\|<\ep/k.
\end{align*}

Note that since $\supp(h^*_0)\le n^+_0$, on the `gap' $(n^+_0,l_1]$ we have that
$\phi^*(n_1)$ is essentially equal to $z^*_0$.  

We proceed inductively in similar fashion. Suppose that for some $1\le j< k$ the statement of the Theorem holds for $\phi_0, \phi(n_1), \ldots, \phi(n_1, \ldots, n_j) $ and $\phi^*_0, \phi^*(n_1), \ldots, \phi^*(n_1, \ldots, n_j) $ where
\begin{align*}
&h_0, h^*_0<n^+_0<l_1<h(n_1), h^*(n_1)<n^+_1<l_2<\cdots <l_{j}<h(n_1, \ldots, n_j), h^*(n_1, \ldots, n_j)<n^+_j
\end{align*}
 are determined
 for $\sU$-positive $n_i$ (in the iterated sense as in the statement of Theorem) and arbitrary choices of $l_i$'s so that 
 \begin{equation} \label{h_j-support}
\begin{aligned}
h(n_1, \ldots, n_i):&=P_{(l_i, n^+_i]}[\phi( n_1, \ldots,  n_{i})],\\ h^*(n_1, \ldots, n_i):&=P^*_{(l_i, n^+_i]}[x^*( n_1, \ldots,  n_{i})],\ i\le j
\end{aligned}
\end{equation}
and we have the estimates 
\begin{equation}\label{norm-estimate-step-j}
\begin{aligned}
&\norm{h_0+\sum_{i=1}^{j} h(n_1, \ldots, n_i)-\phi( n_1, \ldots,  n_{j})}<j\ep/k,
\end{aligned}
\end{equation}
 
\begin{equation}\label{norm-estimate-step-j*}
\norm{\phi^*(n_1, \ldots, n_j)-h^*_0-\sum_{i=1}^j h^*(n_1, \ldots, n_i)-z^*(n_1, \ldots, n_j)}<j\ep/k.
\end{equation} 
 
Let $l_{j+1}>n^+_j$ be arbitrary. Again, using the fact that $\phi( n_1,
\ldots,  n_{j})$ and $\phi^*( n_1, \ldots,  n_{j})$ are simultaneous cluster points of $(\phi( n_1, \ldots, n_{j+1}))_{n_{j+1}}$ 
and  $(\phi^*( n_1, \ldots, n_{j+1}))_{n_{j+1}}$, 
\begin{align*}
H_{j+1}&:=\{n_{j+1}\in\N: |(\phi(n_1,\ldots,n_j)-\phi(n_1,\ldots,n_{j+1}))(e^*_i)|<\frac{\ep}{3kl_{j+1}}, 0<i\le l_{j+1}\} \\
&\cap\{n_{j+1}\in\N:|(x^*(n_1,\ldots,n_j)-x^*(n_1,\ldots,n_{j+1}))(e_i)|<\frac{\ep}{3kl_{j+1}}, 0<i\le l_{j+1}\}\\
&\cap\{n_{j+1}\in\N:|(z^*(n_1,\ldots,n_j)-z^*(n_1,\ldots,n_{j+1}))(e_i)|<\frac{\ep}{3kl_{j+1}}, 0<i\le l_{j+1}\}
\end{align*}
is $\sU$-positive set.
Thus for $\sU$-positive $n_{j+1}>l_{j+1}$

\begin{equation}\label{jth-bump-start}
\begin{aligned}
\|P_{l_{j+1}}[\phi( n_1, \ldots,  n_{j})-\phi( n_1, \ldots,  n_{j+1})]\|&<\frac{\ep}{3k}\ \ \text{and}\\ 
\|P^*_{l_{j+1}}[x^*( n_1, \ldots,  n_{j})-x^*( n_1, \ldots,  n_{j+1})]\|&<\frac{\ep}{3k},\\
\sum_{i=1}^{l_{j+1}}|(\phi^*(n_1,\ldots,n_j)-\phi^*(n_1,\ldots,n_{j+1}))(e_i)|&<\frac{\ep}{3k}
\end{aligned}
\end{equation}
Let $n^+_{j+1}>l_{j+1}$ be such that
\begin{equation}\label{jth-bump-end}
\begin{aligned}
\|P_{n^+_{j+1}}[\phi( n_1, \ldots,  n_{j+1})]-\phi( n_1, \ldots,  n_{j+1})\|<\frac{\ep}{3k}\\
\|P^*_{n^+_{j+1}}[x^*( n_1, \ldots,  n_{j+1})]-x^*( n_1, \ldots,  n_{j+1})\|<\frac{\ep}{3k}.
\end{aligned}
\end{equation}
We put

\begin{align*}
h(n_1,\ldots,n_{j+1}):&=P_{(l_{j+1},n^+_{j+1}]}[\phi( n_1, \ldots,  n_{j+1})]\\
 h^*(n_1,\ldots,n_{j+1}):&=P^*_{(l_{j+1},n^+_{j+1}]}[x^*( n_1, \ldots,  n_{j+1})].
\end{align*}

It is easy to check that we have the desired estimates analogous to
\ref{norm-estimate-step-j} and \ref{norm-estimate-step-j*} for $j+1$. Thus
\ref{permissible-sum} and \ref{permissible-norming} are proved by induction.  

Note that since $l_i<\supp(h^*(n_1, \ldots, n_i))\le n^+_i$, on the gaps
$(n^+_{i-1},l_i]$ we have that $\phi^*(n_1, \ldots, n_i)$ is essentially equal
to $z^*(n_1,\ldots, n_i)$ for all $i\le k$.

We now show \ref{norm-bound-h_j}. Fix $1\le j\le k$. Then by \ref{h_j-support},
\ref{jth-bump-start} and \ref{jth-bump-end} we have
\begin{align*}
\|h(n_1,&\ldots,n_j)\|
=\|P_{n_j^+}[\phi( n_1,\ldots,  n_{j})]- P_{l_j}[\phi( n_1,\ldots,  n_{j})]\|\\
&\le\|P_{n_j^+}[\phi( n_1,\ldots,  n_{j})-\phi( n_1,\ldots,  n_{j-1})] -P_{l_j}[\phi( n_1,\ldots,  n_{j})-\phi( n_1,\ldots,  n_{j-1})]\| +\frac{\ep}{3k}\\
&\le\|P_{n_j^+}[\phi( n_1,\ldots,  n_{j})-\phi( n_1,\ldots,  n_{j-1})]\|+\frac{2\ep}{3k}
\end{align*}
(For $j=k$, replace $P_{n_j^+}[\phi( n_1,\ldots,  n_{j})]$ above by $\phi(
n_1,\ldots,  n_{k})$.) By picking $n_j<m_j$ in the $\sU$-positive set $H_j$ and subsequently picking $n_{l}$ and $m_l$ for $l>j$ from their corresponding $\sU$-positive sets so that   $ n_j<m_j<n_{j+1}<m_{j+1}\ldots   < n_k< m_k$ and   
$$\|P_{n_j^+}[\phi( n_1,\ldots,  n_{j-1},m_{j}, \ldots, m_k)-\phi( n_1,\ldots,
n_{j-1})]\|<\frac{\ep}{6k}.$$

Then continuing the above inequalities
\begin{align*}
&\|P_{n_j^+}[\phi( n_1,\ldots,  n_{j})-\phi( n_1,\ldots,  n_{j-1})]\|+\frac{2\ep}{3k}\\
&\le \|P_{n_j^+}[\phi( n_1,\ldots,  n_{j-1},  n_{j},\ldots,  n_k)-\phi( n_1,\ldots,  n_{j-1}, m_{j}, \ldots, m_k)]\|+\frac{\ep}{k}\\
&\le\|\phi( n_1,\ldots,  n_{j-1},  n_{j},\ldots,  n_k)-\phi( n_1,\ldots,  n_{j-1}, m_{j}, \ldots, m_k)\|+\frac{\ep}{k} \\
&\le K d\Big(( n_1, \ldots,  n_{j-1}, n_j, \ldots, n_k),(n_1, \ldots, n_{j-1},m_j, \ldots, m_k)\Big)+\frac{\ep}{k}.
\end{align*}

This proves \ref{norm-bound-h_j}. Note that we do
 not have a similar estimate for $\|h_0\|$.
\end{proof}

Theorem \ref{linearization} can be stated in terms of the asymptotic structure due to Maurey, Milman, and
Tomczak-Jaegermann \cite{MMT}. We recall this notion with respect to a basis. 
Let $X$ be a separable Banach space with a basis $(e_i)$.
We say that a normalized monotone basis $(u_j)_{j=1}^n$ is {\em an asymptotic
space} for $X$ with respect to $(e_i)$, denoted by $(u_j)_{j=1}^n\in\{X\}_n$, if
for all $\ep>0$, the vector player has a winning strategy in a two player game
of length $n$ where in the $j$th move the subspace player picks a tail subspace
$X_{k_j}=[e_i]_{i\ge k_j}$ and the vector player responds by picking a
normalized vector $x_j\in X_{k_j}$ (we write as $k_j\le x_j$) so that the
resulting sequence $(x_j)_{j=1}^n$ is a block sequence of $(e_i)$ which is $(1+\ep)$-equivalent to $(u_j)_{j=1}^n$.
Thus $(u_i)_{i=1}^n\in \{X\}_n$ with respect to $(e_i)$ if for all
$\ep>0$
\begin{equation*}
\forall k_1\ \exists x_1\ge k_1\ \forall k_2\ \exists x_2\ge k_2 \ \ldots\ \forall k_n\ \exists x_n\ge k_n\ \text{such that}\ (x_j)_{j=1}^n\stackrel{1+\ep}\sim (u_j)_{j=1}^n.
\end{equation*}

Another way of expressing $(u_j)_{j=1}^k\in \{X\}_k$ is that  for all $\ep>0$
there exists a (countable) infinitely branching block tree $\cT_{k}=\{(x_{n_1},
\ldots, x_{n_j})): j\le k\}$ of height $k$ such that each branch $(x_{n_1},
\ldots, x_{n_j})\stackrel{1+\ep}\sim (u_j)_{j=1}^n.$

The collection of $\{X\}_n$'s for all $n$ is referred as {\em the asymptotic
structure} of $X$ (with respect to $(e_i)$), and the block vectors
$x_i$'s (the winning moves of the vector player) are called {\em permissible
vectors}. The subspace player has a winning strategy to play tail subspaces
$X_{k_j}$ so that {\em for every} normalized $x_j\in X_{k_j}$, we have
$(x_j)_{j=1}^n\stackrel{1+\ep}\sim (u_j)_{j=1}^n$ {\em for some}
$(u_j)_{j=1}^n\in \{X\}_n$. Thus if $x_i$'s can be chosen arbitrarily far out we
may assume $(x_i)_{i=1}^n$ is permissible.

As an immediate consequence of Theorem \ref{linearization} by taking $\sU$ to be the Frechet filter we get 

\begin{corollary}\label{lin-asy} Let $X$ be a  Banach space with a boundedly complete basis $(e_i)$. Let $\phi:(\N^k,d)\to X$ be a bounded map
where $d$ is a metric on $\N^k$.

(i) For all $\ep>0$ there exists a block vector $h_0\in X$ for all $l_1>h_0$ there exist $n_1$ and $h(n_1)>l_1$ for all $l_2>h(n_1)$ there exists $n_2$ and $h(n_1,n_2)>l_2$ so on so that for all $l_k>h(n_1, \ldots, n_{k-1})$ there exist $n_k$ and $h(n_1, \ldots, n_k)>l_k$ so that $(h(n_1,\ldots, n_i))_{i=1}^k$ is permissible and

 \begin{align*}
&\norm{\phi(n_1, \ldots, n_k)-h_0-\sum_{i=1}^kh(n_1, \ldots, n_i)}< \ep.\label{permissible-sum}
\end{align*}

(ii) If $\phi$ is $K$-Lipschitz then we have
\begin{align*}
\|h(n_1,\ldots, n_j)\|\le K d((n_1, \ldots, n_k),(m_1, \ldots, m_k))+\ep,\ 1\le j\le k,
\end{align*}

 for some pairs satisfying $n_1=m_1< \ldots< n_{j-1}=m_{j-1}< n_j<m_j<n_{j+1}<m_{j+1}<\ldots <
 n_k<m_k$.

(iii) Moreover,  suppose $X^*=[e^*_i]\oplus
Z^*$ for some $Z^*\subseteq X^*$, and  $\phi^*(n_1, \ldots, n_k)$ be a norming
functional for $\phi(n_1, \ldots, n_k)$. Then for all $\ep>0$ there exist $h^*_0$ and permissible tuple $(h^*(n_1, \ldots, n_i))_{i=1}^k$ with respect to $(e^*_i)$ as in (i) such that 

$$\langle h^*(n_1, \ldots, n_i), h(n_1, \ldots, n_j)\rangle =0,\ i\not= j, \ 1\le i,j\le k,$$ and $z^*_0, z^*(n_1), \ldots, z^*(n_1,\ldots, n_k)\in Z^*$ such that

 \begin{equation*}
\norm{\phi^*(n_1, \ldots, n_k)-h^*_0-\sum_{i=1}^k h^*(n_1, \ldots, n_i)-z^*(n_1, \ldots, n_k)}<\ep
\end{equation*}
and for $j\le k$

$$\Big|\sum_{i=1}^{l_{j+1}}z^*(n_1,\ldots,n_k)(e_i)-\sum_{i=1}^{l_{j+1}}z^*(n_1,\ldots,n_j)(e_i)\Big|<\ep.
$$
 
\end{corollary}

\section{Applications to non-embeddings and rigidity}\label{applications}

We start with applications of Corollary \ref{lin-asy} to deduce the
following known results. 
\begin{itemize}
\item $c_0$, the James space $\J$ and its dual $\J^*$, and, more generally,
non-reflexive Banach spaces with alternating Banach Saks property (for instance,
non-reflexive spaces with type $p>1$) do not coarsely embed into reflexive
spaces (\cite{K}).

\item A separable Banach lattice $X$ is coarsely universal if and only if $c_0$
linearly embeds into $X$ (\cite{K}, see Corollary 5.7 of \cite{BLMS2}). We will
deduce only a special case of this when $X$ has an unconditional basis. 

\item  Let $Y$ be a reflexive asymptotic-$c_0$ Banach space. If $X$ is a Banach
space that coarsely embeds into $Y$, then $X$ is also reflexive and
asymptotic-$c_0$ (\cite{BLMS}).

\end{itemize}

The Kalton's interlacing graphs $(\N^k, d_{\K})$ are modeled on the summing basis
$(s_i)$ of $c_0$. One can generalize them by defining metric spaces  that have a
property which is modeled on arbitrary conditional spreading sequences.  A basis $(x_i)$ is 1-spreading if for all $(a_i)$ and $(n_i)\subseteq \N$ we
have

\begin{equation}\label{def-spreading}
\left\|\sum_{i=1}^{\infty}a_ix_i\right\|=\left\|\sum_{i=1}^{\infty}a_ix_{n_i}\right\|.
\end{equation}

If $(x_i)$ is conditional and spreading, then the summing functional $S(\sum_i
a_ix_i)=\sum_i a_i$ is bounded, and we may assume it has norm 1 (cf.,
\cite{FOSZ}). For basic properties of such sequences that are used below see \cite{FOSZ}, and for a more comprehensive study see \cite{AMS}.

Suppose
$(e_i)$ is a normalized conditional 1-spreading sequence. Then
$(e_{2i}-e_{2i-1})$ is unconditional and not equivalent to the unit vector basis
of $\ell_1$ (otherwise, $(e_i)$ itself would be equivalent to the unit vector
basis of $\ell_1$.) Thus for all $C\ge 1$ there exist positive scalars
$(c_i)_{i=1}^k$ such that $\sum_{i=1}^k c_i\ge C$ while for
$n_1<m_1<n_2<m_2<\ldots<n_k<m_k$ we have  $$\Big\|\sum_{i=1}^k
c_i(e_{n_i}-e_{m_i})\Big\|\le 1.$$   
Since the summing functional $S$ on $[e_i]$ is bounded, we may assume it has norm 1 in particular, and so we have 
$$\Big\|\sum_{i=1}^k c_i e_{n_i}\Big\|\ge  S\Big(\sum_{i=1}^kc_i e_{n_i}\Big)=\sum_{i=1}^kc_i\ge C.$$

\begin{definition}By $\{(\N^k, d^*): k\in\N\}$ denote any family of metric
spaces satisfying the following property.

{\em For all $C\ge 1$ there exists $k\in \N$ such that for all $\vec{n},
 \vec{m}\in \N^{k}$ we have $d^*(\vec{n}, \vec{m})\le 1$ if $\vec{n}, \vec{m}$
 are interlacing, and $d^*(\vec{n}, \vec{m})\ge C$ if
 $\vec{n}<\vec{m}$.}\end{definition}

\begin{remark}
For any increasing function $f:\N\to \R$ with $f(1)=1$ and tending to infinity satisfying $f(n+m)\le f(n)+f(m)$ one can easily define $\{(\N^k, d^*): k\in\N\}$ by taking $\N^k$ as a graph where interlacing tuples are joined, and $d^*$ is a `weighted' shortest distance metric so that the distance is equal to $f(s)$ for nodes $s$ apart. For the function $f(k)=k$, one gets the  Kalton's interlacing graphs. Moreover, it
should be clear from the discussion above that any Banach space with a
conditional spreading model admits an equi-Lipschitz embedding of such a family. This shows that the collection of such families is much broader than simply Kalton's interlacing graphs. For instance, while Kalton's interlacing graphs $(\N^k, d_{\K})$ do not equi-coarsely embed
into the James space (Section \ref{section-J}), James space still admits an equi-Lipschitz embedding of $(\N^k, d^*)$ for some $d^*$ since it has a conditional spreading basis.
\end{remark}

\begin{theorem}\label{ref-asy-unc} Let $X$ be either a reflexive space
 or a space with an asymptotic unconditional boundedly complete basis. Then no family  $\{(\N^k, d^*): k\in\N\}$
 equi-coarsely embeds into $X$.
\end{theorem}

\begin{proof} Let $X$ be a space with an asymptotic unconditional boundedly complete basis $(e_i)$ with the asymptotic unconditionality constant $\theta\ge 1$.
Suppose, for  contradiction, that there is a family $\{(\N^k, d^*): k\in\N\}$
and equi-coarse embeddings $$\phi_k:(\N^k, d^*)\to X.$$ Then we have a
non-decreasing function $\rho$ with $\lim_{t\to\infty}\rho(t)=\infty$ and
function $\omega(t)<\infty$ for all $t\in[0,\infty)$ so that for all $\vec{n},
\vec{m}\in \N^k$ and $k\in\N$ we have $$\rho(d^*(\vec{n}, \vec{m}))\le
\|\phi_k(\vec{n})-\phi_k(\vec{m})\|\le \omega( d^*(\vec{n}, \vec{m})).$$

Let $C\ge 8\theta\omega(1)$ and $k\in\N$ be such that $\rho(d^*(\vec{n},
\vec{m}))\ge C$ for all $\vec{n}<\vec{m}$ , and $\rho(d^*(\vec{n},
\vec{m}))\le \omega(1)$ for all interlacing tuples $\vec{n},\vec{m}$ in $\N^k$.

From now on we only work with $\phi_k$ so we drop the subscript $k$ and write
$\phi:=\phi_k$.  
Let $\ep>0$.  By Corollary
\ref{lin-asy} there exists a full subtree $\cT$ (that is, every node has infinitely many immediate successors) of $\N^k$ such that for all $(n_1,
n_2,\ldots, n_k)\in \cT$ we have 

\begin{align*}
&\Big\|\phi(\vec{n})-h_0-\sum_{i=1}^{k}h(n_1, \ldots, n_i)\Big\|< \ep
\end{align*}
for some permissible block vectors $h_0<h(n_1)<\ldots<h(n_1,\ldots,n_k)$ with
respect to $(e_i)$.
 
By Ramsey\footnote{The version of Ramsey we use here is that any partition of a full tree of height $k$ into $l$ many colors has a homogeneous full subtree.} there exists a full subtree $\cT'\subseteq \cT$
so that that for all $(n_1, \ldots, n_k)\in \cT'$ we have 
$$\left|\Big\|\sum_{i=1}^k h(n_1, \ldots, n_i)\Big\|-\eta\right|<\ep$$ for some $\eta.$ We estimate $\eta$. Let $\vec{n}<\vec{m}$ be successive $k$-tuples in $\cT'$. We have
\begin{align*}
8\theta\omega(1)<C&\le \|\phi(\vec{n})-\phi(\vec{m})\|\\
&\le \Big\|\sum_{i=1}^k h(m_1, \ldots, m_i)-\sum_{i=1}^k h(n_1, \ldots, n_i)\Big\|+2\ep\\
&\le \Big\|\sum_{i=1}^k h(m_1, \ldots, m_i)\Big\|+\Big\|\sum_{i=1}^k h(n_1, \ldots, n_i)\Big\|+2\ep\\
&\le 2\eta+4\ep.\\
\end{align*}
That is, $\eta\ge 4\theta\omega(1)-2\ep$.

On the other hand, applying Corollary \ref{lin-asy} to a pair of strictly
interlacing tuples $ n_1 <  m_1 < n_2< m_2<\ldots< n_k < m_k$ in $\cT'$ we have similar estimates for $\phi(\vec n)$ and $\phi(\vec m)$ in terms of
permissible block vectors of the form
$h_0<h(n_1)<h(m_1)<h(n_1,n_2)<h(m_1,m_2)<\ldots<h(n_1, \ldots, n_k)<h(m_1, \ldots, m_k)$. Then 
\begin{align*}
4\theta\omega(1)-2\ep&\le \Big\|\sum_{i=1}^{k}h(n_1, \ldots, n_i)\Big\|
\le \theta \Big\|\sum_{i=1}^{k}h(n_1, \ldots, n_i)-\sum_{i=1}^{k}h(m_1, \ldots, m_i)\Big\|\\&\le \theta \|\phi(\vec{n})-\phi(\vec{m})\|\le \theta\omega(1),
\end{align*}
which yields a contradiction, and completes the proof.

In the case of
reflexive $X$, we may assume $X$ is separable and embeds into a reflexive space
with a boundedly complete basis $(e_i)$. Then Corollary \ref{lin-asy} yields a
permissible $k$-tuple of functionals $(h^*(n_1, \ldots, n_i))_{i=1}^{k}$ vanishing on $h(m_1), \ldots, h(m_1, \ldots, m_k)$ such that
\begin{align*}
&\Big\|\sum_{i=1}^{k} h(n_1, \ldots, n_i)\Big\|\le (1+\ep)\Big<\sum_{i=1}^{k} h^*(n_1, \ldots, n_i),\sum_{i=1}^{k} h(n_1, \ldots, n_i)\Big>,\ \text{and}\\
&\Big\|\sum_{i=1}^{k} h^*(n_1, \ldots, n_i)\Big\|\le 1+\ep.
\end{align*}

Repeat the first part of the proof above (the stabilization argument) and assume $k$ is large enough so that the first inequality below holds. Then putting these together
\begin{align*}
4\omega(1)-2\ep&\le \Big\|\sum_{i=1}^{k}h(n_1, \ldots, n_i)\Big\|
\le (1+\ep)\left<\sum_{i=1}^{k}h^*(n_1, \ldots, n_i), \sum_{i=1}^{k}h(n_1, \ldots, n_i)\right>\\
&=(1+\ep)\left<\sum_{i=1}^{k}h^*(n_1, \ldots, n_i), \sum_{i=1}^{k}h(n_1, \ldots, n_i)-\sum_{i=1}^{k}h(m_1, \ldots, m_i)\right>\\
& \le (1+\ep)^2\Big\| \sum_{i=1}^{k}h(n_1, \ldots, n_i)-\sum_{i=1}^{k}h(m_1, \ldots, m_i)\Big\|\\
&\le(1+\ep)^2\|\phi(\vec{n})-\phi(\vec{m})\|+(1+\ep)^22\ep\\
&\le(1+\ep)^2\omega(1)+3\ep,
\end{align*}
  yields a contradiction for small enough $\ep$.
 
\end{proof}

By Corollary 10.5 of \cite{AMS} a non-reflexive space $X$ with no $\ell_1$
spreading models admits a conditional spreading model. Thus we have

\begin{corollary}\label{non-ref-no-l_1}
Let $X$ be a non-reflexive Banach space which admits no $\ell_1$ spreading
models. Let $Y$ be either reflexive or has asymptotic unconditional boundedly
complete basis. Then $X$ does not coarsely embed into $Y$.
\end{corollary}

\begin{remark} Thus if a non-reflexive $X$ coarsely embeds into $\ell_2$, then $X$ must  have $\ell_1$ spreading models. The reader should also recall the well known fact that $\ell_1$ coarsely embeds into $\ell_2$. 

\end{remark}

Recall the classical fact that $c_0$ linearly embeds into a space with an
unconditional basis if and only if the basis is not boundedly complete. Thus we have
 
\begin{corollary}
Let $X$ be a Banach space with an unconditional basis. Then $c_0$ coarsely
embeds into $X$ if and only if $c_0$ linearly embeds into $X$.
\end{corollary}

Another consequence of Corollary \ref{lin-asy} is the coarse rigidity
of reflexive asymptotic-$c_0$ spaces \cite{BLMS}. As in the original proof, we make use of a characterization of asymptotic-$c_0$ spaces via {\em asymptotic models} given in \cite{FOSZ}.

Let $(x^j_i)_{i\ge 1}, 1\le j\le k$ be an array of $C$-basic  sequences in a space $X$. The arrays could be infinite but we only consider finite ones. We say that $(e_j)_{j=1}^k$ is an asymptotic model generated by the $C$-basic array $(x^j_i)$ if for all scalars $(a_j)_{j=1}^k$ we have 
$$\lim_{i_1\to\infty}\ldots \lim_{i_k\to\infty}\Big\|\sum_{j=1}^k a_j x^j_{i_k}\Big\|=\Big\|\sum_{j=1}^k a_j e_j\Big\|.$$ 
That is, the norm of the linear combinations of diagonal elements in the array stabilize. As an application of Ramsey's theorem, every $C$-basic array has a subarray that generates an asymptotic model. For details, we refer to Halbeisen and Odell \cite{HO} where this notion was introduced and studied.

Asymptotic spaces are generated by countably branching trees while asymptotic models are generated by arrays. Therefore, the following result of \cite{FOSZ} came as a surprise.

\begin{theorem}[FOSZ]
Suppose that a Banach space $X$ does not contain an isomorphic
copy of $\ell_1$ and every asymptotic model $(e_i)$ generated by weakly null arrays in $X$ is
equivalent to the unit vector basis of $c_0$. Then

(i) $X^*$ is separable, and thus $X$ embeds into a space $Y$ with a shrinking basis $(yi)$.

(ii) $X$ is asymptotic-$c_0$ (with respect to the basis $(y_i)$).
\end{theorem}

As in \cite{BLMS}, we use this to prove:

\begin{theorem}\label{asy-c_0} Let $Y$ be a separable reflexive asymptotic-$c_0$ Banach space. If
$X$ is a Banach space that coarsely embeds into $Y$, then $X$ is also reflexive
and asymptotic-$c_0$.
\end{theorem}

\begin{proof}
Suppose $\phi:X\to Y$ is a coarse embedding such that for some constant $K$ and
a map $\rho$ with $\lim_{t\to\infty}\rho(t)=\infty$ we have

$$\rho(\|x-y\|)\le\|\phi(x)-\phi(y)\|\le K \|x-y\|,$$ 
for $\|x-y\|\ge 1$. Suppose that $Y$ is reflexive and asymptotic-$c_0$ with
constant $C$.

 If $X$ were non-reflexive, by James' characterization of reflexivity it contains
 a $\ell^+_1$ sequence. That is, there exists an infinite sequence $(x_i)$ in
 the unit ball such that for all $k$ and all $n_1<\ldots<n_k<m_1<\ldots<m_k$ in
 $\N^{2k}$ we have
$$\Big\|\sum_{i=1}^k x_{n_i}-\sum_{i=1}^k x_{m_i}\Big\|\ge \frac{k}{2}.$$  

Fix $k$ such that $\rho(k/2)\ge 8KC$. Thus for  $n_1<\ldots<n_k<m_1<\ldots<m_k$  we have
$$ \left\|\phi\Big(\sum_{i=1}^k x_{n_i}\Big)-\phi\Big(\sum_{i=1}^k
x_{m_i}\Big)\right\|\ge 8KC.$$

$Y$ embeds into a reflexive space $Y'$ with a basis $(e_i)$. We take this basis $(e_i)$ in Corollary \ref{lin-asy} apply to the map 
$\phi(\vec{n}):=\phi\Big(\sum_{i=1}^k x_{n_i}\Big)$ where the tuples are from
the sequence $(x_i)$ we get a subsequence $M\subset \N$ so that for all $(n_1, \ldots, n_k)\in M^k$ we have (suppressing tiny approximations)
$$\phi\Big(\sum_{i=1}^k x_{n_i}\Big)=h_0 +\sum_{i=1}^k h(n_1, \ldots, n_i)$$ for some block vectors $h_0<h(n_1)<\ldots<h(n_1,\ldots, n_k)$ (where the block structure is with respect to the basis $(e_i)$).

We claim that for all $n_1< \ldots< n_k\in M$ with at least one $m_i\in M$ in between $n_i<m_i<n_{i+1}$, we have
 $$\|h(n_1, \ldots, n_i)\|\le 2K$$ for all $1\le i\le
k$.

Indeed, fix $1\le i\le k$. For a given such tuple $n_1<\ldots< n_k\in M$ let $m_j=n_j$ for all $i\neq j$, and $m_i\in M$ with $n_i<m_i<n_{i+1}$. (For $i=k$, pick $n_{k-1}<m_{k-1}<n_k$.) Then
\begin{align*}&\|h(n_1, \ldots, n_i)\|\le \|h(n_1, \ldots, n_i)-h(m_1, \ldots, m_i)\|\\ &\le \Big\|\sum_{j\ge i}h(n_1, \ldots, n_j)-h(m_1, \ldots, m_j)\Big\|=\Big\|\phi\Big(\sum_{i=1}^k x_{n_i}\Big)-\phi \Big(\sum_{i=1}^k x_{m_i}\Big)\Big\|\\
&\le K\|x_{n_i}-x_{m_i}\|\le 2K.
\end{align*}
The first and the second inequalities above follow from the fact that the blocks involved are permissible (thus $(1+\ep)$-basic, which we suppressed the error for brevity), and the first $i-1$ blocks of both tuples are identical and hence cancel out. 

Now let $n_1<\ldots<n_k<m_1<\ldots<m_k$ be two tuples in $M$ with at least one gap in $M$ between coordinates, By the above claim, the corresponding blocks  $h(n_1)<\ldots<h(n_1,\ldots, n_k)<h(m_1)<\ldots<h(m_1,
\ldots, m_k)$ are permissible in $Y$ and 
$$\|h(n_1, \ldots, n_i)\|, \|h(m_1, \ldots, m_i)\|\le 2K$$ for all $1\le i\le
k$. Therefore, by our initial assumption and the main assumption that $Y$ is asymptotic-$c_0$, we have
\begin{align*}
8KC \le& \left\|\phi\Big(\sum_{i=1}^k x_{n_i}\Big)-\phi\Big(\sum_{i=1}^k x_{m_i}\Big)\right\|\\
&\le \left\|\sum_{i=1}^{k}h(n_1, \ldots, n_i)-\sum_{i=1}^{k}h(m_1, \ldots, m_i)  \right\|+2\ep,\\
&\le2KC+2\ep,
\end{align*}
which is a contradiction. Thus $X$ must be reflexive. 

Now we show that $X$ is asymptotic-$c_0$.
Note that, since $X$ is reflexive, the asymptotic structure is independent of the filter
used. Since $X$ embeds into a reflexive space with a basis we will take the asymptotic structure with respect to the basis of the super space. For the sake of contradiction suppose that $X$ is not asymptotic-$c_0$. By the theorem of \cite{FOSZ} mentioned above, 
there exist asymptotic models $(e_i)_{i=1}^k$ of $X$ generated by normalized weakly null arrays (which are 1-basic) with $\|\sum_{i=1}^ke_i\|\nearrow \infty$ as
$k\nearrow \infty$.\footnote{This follows from the fact that asymptotic models generated by weakly null arrays are 1-suppression unconditional, see \cite{HO}.} Let $k$ be
sufficiently large so that $\rho(k)>4KC$ and $\ep>0$, then there exists a normalized 1-basic weakly null array $(x^j_i)_{i\ge 1, 1\le j\le k}$ in
$X$ such that for all $(x^1_{n_1}, \ldots, x^k_{n_k})$ with $n_1<\ldots< n_k$ we have $(x^i_{n_i})_{i=1}^k\stackrel{1+\ep}\sim (e_i)_{i=1}^k$ and 
$$\Big\|\phi\Big(\sum_{i=1}^k x^i_{n_i}\Big)\Big\|\ge 4KC.$$

Again, applying Corollary \ref{lin-asy} to
$\phi(\vec{n}):=\phi\Big(\sum_{i=1}^k x^i_{n_i}\Big)$ where the tuples are from
the array $(x^j_i)_{i\ge 1, 1\le j\le k}$ we get a subarray $(x^j_i)_{i\in M, 1\le j\le k}$  so that for all $(n_1, \ldots, n_k)$ we have (suppressing tiny approximations)
$$\phi\Big(\sum_{i=1}^k x^i_{n_i}\Big)=h_0 +\sum_{i=1}^k h(n_1, \ldots, n_i)$$ for some block vectors $h_0<h(n_1)<\ldots<h(n_1,\ldots, n_k)$ (where the block structure is with respect to the basis $(e_i)$ of $Y'\supseteq Y$ as in the first part of the proof).

Consider the tuples $(n_1,\ldots, n_k)$ with at least one gap in $M$ between each consecutive coordinates.  With an essentially identical argument as in the first part of the proof, for all such tuples $(n_1,\ldots, n_k)$  we have 
$$\|h(n_1,\ldots, n_i)\|\le 2K, 1\le i\le k.$$

Since the array $(x^j_i)_{i\in M, 1\le j\le k}$ is weakly null, there exist tuples  $n_1<\ldots n_k <m_1<\ldots<m_k$ with gaps so that

$$\Big\|\sum_{i=1}^k x^i_{n_i}\Big\|\le (1+\ep)\Big\|\sum_{i=1}^k
x^i_{n_i}-\sum_{i=1}^k x^i_{m_i}\Big\|.$$

Applying the Corollary \ref{lin-asy}, we get a pair of tuples as above so
that on one hand we still have
$$\left\|\phi\Big(\sum_{i=1}^k x^i_{n_i}\Big)-\phi\Big(\sum_{i=1}^k
x^i_{m_i}\Big)\right\|\ge 4KC,$$
and on the other hand we have
\begin{align*}
&\Big\|\phi\Big(\sum_{i=1}^k x^i_{n_i}\Big)-h_0-\sum_{i=1}^{k}h(n_1, \ldots, n_i)\Big\|< \ep,\\
&\Big\|\phi\Big(\sum_{i=1}^k
x^i_{m_i}\Big)-h_0-\sum_{i=1}^{k}h(m_1, \ldots, m_i)\Big\|< \ep,
\end{align*}
for some permissible vectors $h(n_1)<\ldots<h(n_1,
\ldots, n_i)<h(m_1)<\ldots<h(m_1, \ldots, m_k)$ in $Y$ such that 
$$\|h(n_1, \ldots, n_i)\|, \|h(m_1, \ldots, m_i)\|\le 2K, \  i\ge 1.$$
Thus
\begin{align*}
4KC&\le \left\|\phi\Big(\sum_{i=1}^k x^i_{n_i}\Big)-\phi\Big(\sum_{i=1}^k x^i_{m_i}\Big)\right\|\\
&\le \left\|\sum_{i=1}^{k}h(n_1, \ldots, n_i)-\sum_{i=1}^{k}h(m_1, \ldots, m_i)  \right\|+2\ep\\
&\le 2KC+2\ep,
\end{align*}
a contradiction for small $\ep>0$. The last inequality again uses the fact that $Y$ is $C$-asymptotic-$c_0$.

\end{proof}

\section{Coarse non-universality of dual spaces with spreading bases} \label{sec:spreading}

In this section we prove coarse non-universality of dual Banach spaces with a conditional spreading basis. An important ingredient of the proof is an infinite pigeonhole argument via Ramsey ultrafilters.   
Recall that by Lemma~\ref{comp} and Fact~\ref{rufact} we may assume without loss of generality that 
there is a Ramsey ultrafilter.
This principle is also used in the proofs of Sections \ref{quasi-reflexive} and \ref{section-gen-james}. 
The infinite pigeonhole argument used in Theorem~\ref{spreading}
is used to show non-embedding of $c_0$ in spaces where  the stronger statement of non-embedding of the Kalton graphs may not hold (see Remark \ref{motakis}).

The definition of spreading bases and a brief discussion of conditional ones were recalled around equation \ref{def-spreading}. Additionally, if $(x_i)$ is a conditional spreading basis for $X$ then $(x_i)$ is boundedly
complete if and only if $c_0$ does not linearly embed into $X$ (Theorem 2.3,
\cite{FOSZ}). For instance, the boundedly complete basis of the James space is
conditional spreading.

We will make use of the following lemma. This is essentially the proof of the
fact that if $(u_i)$ is a block basis of a conditional spreading basis $(x_i)$
and $S(u_i)=0$ for all $i$, then $(u_i)$ is suppression 1-unconditional, that is, for any $n$ and any subset $A\subseteq \{1, \ldots, n\}$, and scalars $(a_i)$ we have 
$\|\sum_{i\in A}a_i x_i\|\le \|\sum_{i=1}^n a_i x_i\|$. (Lemma 2.4,
\cite{FOSZ}).

Let $l_1<\ldots<l_s$, $t_1<\ldots<t_s$, and $u=\sum_{i}^s a_ix_{l_i}$ be a
finitely supported vector.  We write $u\sim u'$ if $u'=\sum_{i=1}^s a_ix_{t_i}$.
We call $u'$ {\em a spread of} $u$. Since $(x_i)$ is 1-spreading $\|u\|=\|u'\|$.

\begin{lemma}\label{spreads} Assume $(x_i)$ is a basis such that the summing functional has norm 1. Let $\ep>0$ and $s, n\in \N$. Then there exists
$m>n$ such that for all $f\in X^*$ with $\|f\|\le 1$ and all $u=\sum_{i=1}^s a_ix_{l_i}$ with
$\|u\|=1$ and $F=\{l_1<\ldots<l_s\}$ there exists
$F'=\{t_1<\ldots<t_s\}\subseteq [n,m]$ and $\lambda\in [-1,1]$ such that if
$u'\sim u$ with $u'=\sum_{i=1}^s a_ix_{t_i}$ then 
$$\left|f(u')-\lambda\sum_{i=1}^s a_i\right|<\ep.$$
\end{lemma}

\begin{proof}
Let $n\in \N$. For all $f\in X^*$ with $\|f\|\le 1$, we have $|f(x_i)|\le 1$. Therefore, by the pigeonhole principle there exists $m$ with the following
property:

For all $f\in X^*$ with $\|f\|\le 1$ there exists $\lambda\in [-1,1]$ and
$F'=\{t_1<\ldots<t_s\}\subseteq [n,m]$ so that for all $1\le i\le s$ we have
$|f(x_{t_i})-\lambda|<\ep$. Then for $u'=\sum_{i=1}^s a_ix_{t_i}$ we have 
$$f(u')=f\Big(\sum_{i=1}^s a_ix_{t_i}\Big)=\sum_{i=1}^s a_if(x_{t_i}),$$ hence
the result follows since $|\sum_{i=1}^s a_i|\le 1$.
\end{proof}

\begin{theorem}\label{spreading} Assume $X$ is a dual space with a conditional spreading basis $(x_i)$. Then $c_0$ does not coarsely embed into $X$.
\end{theorem}

\begin{proof}

Suppose there is a coarse embedding $\phi:c_0\to X$. Then there exists a constant
$K>0$ and a function $\rho$ with $\lim_{t\to \infty}\rho(t)=\infty$ such that
for all $x,y$ with $\|x-y\|\ge 1$ we have

$$\rho(\|x-y\|)\le \|\phi(x)-\phi(y)\|\le K\|x-y\|.$$

On the other hand, since $X$ is a dual space, there is no linear embedding of $c_0$. By our remarks at the beginning of this section, the basis $(x_i)$ must be boundedly complete.

For every infinite $A=(l_j)\subseteq \N$ and $\vec{n} =(n_1, \ldots, n_k)\in
\N^k$ put
 
$$\phi^A(\vec{n})=\phi\left(\sum_{i=1}^k s_{n_i}(A)\right)$$ where
$s_n(A)=\sum_{j=1}^ne_{l_j}$ and $(e_j)$ is the unit vector basis of $c_0$.

Let $\sU$ be non-principal Ramsey ultrafilter. Let $\ep>0$, $k\in \N$. By
Theorem \ref{linearization} applied to each $\phi^A$ there exists $H'_A\in \sU$
and $h^A_0\in X$ such that for all $\vec{n} \in {H'_A}^k$ we have permissible
(with respect to $(x_i)$) $(h^A(n_1) \ldots, h^A(n_1, \ldots, n_k))$ satisfying
$$\Big\|\phi^A(\vec{n})-h^A_0-\sum_{i=1}^kh^A(n_1, \ldots, n_i)\Big\|< \ep.$$

Let $S\in X^*$ be the summing functional. Since $\sU$ is Ramsey, by a partition
argument, there exist $H_A\subseteq H'_A$ in $\sU$ such that for all
$n_1<\ldots<n_k, n'_1<\ldots<n'_k\in H_A^k$ we have 

\begin{equation*}
\sum_{i=1}^k\Big|S\Big( h^A(n_1, \ldots, n_i)-h^A(n'_1, \ldots, n'_i)\Big)\Big|<\ep.
\end{equation*}

Since $X$ is separable and the collection $(\phi^A)_{A\subseteq \N}$ is
uncountable, by the pigeonhole principle there exists uncountable $\mathcal C$
such that for all $A,B\in \mathcal C$ and for all $(n_i)_{i=1}^k \in (H_A\cap
H_B)^k$ we have

\begin{equation}\label{same sums}
\begin{aligned}
\|h^A_0-h^B_0\|\le \ep/k,\ \text{and}&\\
\sum_{i=1}^k\Big|S\Big(h^A(n_1, \ldots, n_i)-h^B(n_1, \ldots, n_i)\Big)\Big|&<\ep.
\end{aligned}
\end{equation}

Now pick distinct $A, B\in \mathcal C$. Then for all 
$\vec{n}=(n_1,\ldots, n_k)\in (H_A\cap H_B)^k$ large enough that $n_1\geq\min(A\Delta B)$ we have 
\begin{equation}\label{rho k}
\|\phi^A(\vec{n})-\phi^B(\vec{n})\|\ge \rho(k)
\end{equation}
since $\|\sum_{i=1}^k s_{n_i}(A)-\sum_{i=1}^k s_{n_i}(B)\|_{c_0}=k$ for such vectors.

Let $\vec{m}=(m_1,\ldots,m_k)$ and $\vec{n}=(n_1,\ldots,n_k)$ be sufficiently
spread out and large enough as above such that $m_1<n_1<\ldots<m_k<n_k\in
(H_A\cap H_B)^{2k}$ and let 
\begin{align}\label{A-B-interlacing}h^A(m_1), h^B(m_1)<h^A(n_1), h^B(n_1)&<\nonumber\\\ldots<h^A(m_1,\ldots, m_k),  &h^B(m_1,\ldots, m_k)
<h^A(n_1, \ldots, n_k), h^B(n_1, \ldots, n_k)
\end{align} 
be the permissible block vectors as in Theorem \ref{linearization} for
$\phi^A(\vec{m}), \phi^B(\vec{m})$ and $\phi^A(\vec{n}),\phi^B(\vec{n})$ whose
supports are in indicated order.

Let $f\in X^*$ with $\|f\|\le 1$ with
$$f(\phi^A(\vec{n})-\phi^B(\vec{n}))=\|\phi^A(\vec{n})-\phi^B(\vec{n})\|.$$ Let
$h'^A(m_1,\ldots, m_i)-h'^B(m_1,\ldots, m_i)$'s be the spreads of
$h^A(m_1,\ldots, m_i)-h^B(m_1,\ldots, m_i)$'s with 
\begin{align*}h^A&(m_1)-h^B(m_1)<h'^A(m_1)-h'^B(m_1)<h^A(n_1)-h^B(n_1)<\ldots\\ &<h^A(m_1,\ldots,m_k)-h^B(m_1,\ldots,m_k)<h'^A(m_1,\ldots, m_k)-h'^B(m_1,\ldots, m_k)\\&<h^A(n_1, \ldots, n_k)-h^B(n_1, \ldots, n_k)
\end{align*} 
for the functional $f$ as in Lemma \ref{spreads}. By \ref{same sums}
$$
 \sum_{i=1}^k \Big|S\Big( h'^A(m_1,\ldots, m_i)-h'^B(m_1,\ldots, m_i)\Big)\Big|= \sum_{i=1}^k \Big|S\Big( h^A(m_1,\ldots, m_i)-h^B(m_1,\ldots, m_i)\Big)\Big|<\ep,$$ and by Lemma \ref{spreads} for some $(\lambda_i)_{i=1}^k\in [-1,1]^k$ we have
 
$$\left|f\left(\sum_{i=1}^k  h'^A(m_1,\ldots, m_i)-h'^B(m_1,\ldots,
m_i)\right)\right|\le \sum_{i=1}^k |\lambda_i|\big| S (h'^A(m_1,\ldots,
m_i)-h'^B(m_1,\ldots, m_i))\big|+\ep<2\ep.$$

Then, suppressing the approximations of Theorem \ref{linearization},

\begin{align*}
&\|\phi^A(\vec{n})-\phi^B(\vec{n})\|=f\left(\phi^A(\vec{n})-\phi^B(\vec{n})\right)\\
&\le f\left(\sum_{i=1}^k h^A(n_1, \ldots, n_i)-\sum_{i=1}^k h^B(n_1, \ldots, n_i)\right)+\ep\\
&\le f\left(\sum_{i=1}^k h^A(n_1, \ldots, n_i)-\sum_{i=1}^k h'^A(m_1,\ldots, m_i)\right)+f\left(-\sum_{i=1}^k h^B(n_1, \ldots, n_i)+\sum_{i=1}^k h'^B(m_1,\ldots, m_i)\right)\\
&\ \ \ \ \ \ +f\left(\sum_{i=1}^k h'^A(m_1,\ldots, m_i)-\sum_{i=1}^k h'^B(m_1,\ldots, m_i)\right)+\ep\\
&\le\left\|\sum_{i=1}^k h^A(n_1, \ldots, n_i)-\sum_{i=1}^k h'^A(m_1,\ldots, m_i)\right\|+\left\|\sum_{i=1}^k h^B(n_1, \ldots, n_i)-\sum_{i=1}^k h'^B(m_1,\ldots, m_i)\right\|+3\ep\\
&=\left\|\sum_{i=1}^k h^A(n_1, \ldots, n_i)-\sum_{i=1}^k h^A(m_1, \ldots, m_i)\right\|+\left\|\sum_{i=1}^k h^B(n_1, \ldots, n_i)-\sum_{i=1}^k h^B(m_1, \ldots, m_i)\right\|+3\ep\\
&\le\|\phi^A(\vec{n})-\phi^A(\vec{m})\|+\|\phi^B(\vec{m})-\phi^B(\vec{n})\|+3\ep\\
&\le 2K+3\ep,
\end{align*}
which contradicts \ref{rho k} as $\rho(k)\to\infty$.
\end{proof}

\section{Coarse non-universality of quasi-reflexive spaces}\label{quasi-reflexive}
In this section, we show that a modification of the proof from the previous
section works if we replace the spreading basis assumption on $X$ by a `small
dual' assumption, that is, $X^*=[e^*_i]\oplus Z^*$ for some finite dimensional
space $Z^*$.
 
\begin{theorem}\label{quasireflexive} Suppose $X$ has boundedly complete basis
$(e_i)$ and $X^*=[e^*_i]\oplus Z^*$ for some finite dimensional space $Z^*$.
Then $c_0$ does not coarsely embed into $X$.
\end{theorem}

\begin{proof}
For every infinite $A=(l_j)\subseteq \N$ and $(n_1, \ldots, n_k)\in \N^k$ put
 
$$\phi^A(n_1, \ldots, n_k)=\phi\left(\sum_{i=1}^k s_{n_i}(A)\right)$$ where
$s_n(A)=\sum_{j=1}^ne_{l_j}$ and $(e_j)$ is the unit vector basis of $c_0$.

Let $\sU$ be a non-principle Ramsey ultrafilter. Let $\ep>0, k\in\N$. By Theorem
\ref{linearization} applied to each $\phi^A$ there exist $H_A\in \sU$, $h^A_0$ so
that for all $(n_1, \ldots, n_k)\in H_A^k$ there exists blocks $(h^A(n_1), \ldots,
h^A(n_1, \ldots, n_k))\in X$ so that
$$\|\phi^A(n_1,\ldots, n_k)-\sum_{i=0}^k h^A(n_1, \ldots, n_i)\|<\ep,$$ 

Let $(z^*_i)_{i=1}^m$ be a finite $\ep$-net in $B_{Z^*}$. Since $\sU$ is Ramsey, for each $A\subseteq \N$
there exists $H_A'\in \sU, H_A'\subseteq H_A$ so that  for all $(m_i)_{i=1}^k,
(n_i)_{i=1}^k\in H_A'$ and all $1\le i\le m$ we have 
$$|z^*_i(\sum_{i=0}^kh^A(n_1, \ldots, n_i)-\sum_{i=0}^kh^A(m_1,\ldots,m_i))|<\ep.$$

Since $X$ is separable and the collection $(\phi^A)_{A\subseteq \N}$ is
uncountable, by the pigeonhole principle there exists uncountable $\mathcal{C}$ such that
for all $A,B\in \mathcal{C}$ 
$$\|h^A_0-h^B_0\|\le \ep,$$ and for all $(n_i)_{i=1}^k\in (H_A'\cap H_B')^k$ we have 

$$|z^*_i(\sum_{i=0}^kh^A(n_1, \ldots, n_i)-\sum_{i=0}^kh^B(n_1, \ldots,
n_i))|<\ep,$$ for all $1\le i\le m$.

Now pick distinct $A, B\in \mathcal{C}$. Then for all $\vec n=(n_1, \ldots,
n_k)\in (H_A'\cap H_B')^k$ with $n_1\geq \min(A\Delta B)$ we have 
\begin{equation}\label{rho k}
\|\phi^A(\vec{n})-\phi^B(\vec{n})\|\ge \rho(k)
\end{equation}
since $\|\sum_{i=1}^k s_{n_i}(A)-\sum_{i=1}^k s_{n_i}(B)\|_{c_0}=k$ for such vectors.

Let $\vec{m}=(m_1<\ldots<m_k)$ and $\vec{n}=(n_1<\ldots<n_k)$ be such that
$m_1<n_1<\ldots<m_k<n_k\in (H_A'\cap H_B')^{2k}$ and where

\begin{align*}h^A(m_1)-h^B(m_1)<h^A(n_1)-h^B(n_1)<\ldots<h^A(m_1,\ldots,
m_k)-\\h^B(m_1,\ldots, m_k)<h^A(n_1, \ldots, n_k)-h^B(n_1, \ldots, n_k)
\end{align*} be the
block vectors as in Theorem \ref{linearization} for $\phi^A(\vec{m}),
\phi^A(\vec{n})$ and $\phi^B(\vec{n}), \phi^B(\vec{m})$. 

Let $f\in X^*$ with $\|f\|\le 1$ with
$$f(\phi^A(\vec{n})-\phi^B(\vec{n}))=\|\phi^A(\vec{n})-\phi^B(\vec{n})\|.$$ 

By part (iii) of Theorem \ref{linearization} applied to $\phi^A(\vec{n})-\phi^B(\vec{n})$,
there exists functionals $(u^*_{n_i})_{i=1}^k$ so that
$u^*_{n_i}(h_{m_j}(A)-h_{m_j}(B))=0$ for all $1\le i, j\le k$ (since
$u^*_{n_i}$'s are supported around $n_i$'s) and $z^*_{i_0}$ for some $1\le
i_0\le m$ so that we may take $f$ of the form
 
$$f=\sum_{i=1}^ku^*_{n_i}+z^*_{i_0}.$$

(Note that since $h^A_0-h^B_0$ is approximately zero we may take $u^*_0$ to be
zero.) Thus, we have 

$$\left\langle f,  \sum_{i=1}^k h_{m_i}(A)-\sum_{i=1}^k
h_{m_i}(B)\right\rangle=\left\langle \sum_{i=1}^ku^*_{n_i}+z^*_{i_0},
\sum_{i=1}^k h_{m_i}(A)-\sum_{i=1}^k h_{m_i}(B)\right\rangle \leq \ep$$

Then, suppressing those approximations which come from Theorem \ref{linearization},
\begin{align*}
&\|\phi^A(\vec{n})-\phi^B(\vec{n})\|=f\left(\phi^A(\vec{n})-\phi^B(\vec{n})\right)\\
&=f\left(\sum_{i=1}^k h^A(n_1, \ldots, n_i)-\sum_{i=1}^k h^B(n_1, \ldots, n_i)\right)+\ep\\
&=f\left(\sum_{i=1}^k h^A(n_1, \ldots, n_i)-\sum_{i=1}^k h^A(m_1,\ldots,m_i)\right)+f\left(-\sum_{i=1}^k h^B(n_1, \ldots, n_i)+\sum_{i=1}^k h^B(m_1,\ldots,m_i)\right)\\
&+f\left(\sum_{i=1}^k h^A(m_1,\ldots, m_i)-\sum_{i=1}^k h^B(m_1,\ldots, m_i)\right)+\ep\\
&\le\left\|\sum_{i=1}^k h^A(n_1, \ldots, n_i)-\sum_{i=1}^k h^A(m_1,\ldots,m_i)\right\|+\left\|\sum_{i=1}^k h^B(n_1, \ldots, n_i)-\sum_{i=1}^k h^B(m_1,\ldots, m_i)\right\|+2\ep\\
&=\left\|\sum_{i=1}^k h^A(n_1, \ldots, n_i)-\sum_{i=1}^k h^A(m_1,\ldots,m_i)\right\|
+\left\|\sum_{i=1}^k h^B(n_1, \ldots, n_i)-\sum_{i=1}^k h^B(m_1,\ldots,m_i)\right\|+2\ep\\
&\le\|\phi^A(\vec{n})-\phi^A(\vec{m})\|+\|\phi^B(\vec{m})-\phi^B(\vec{n})\|+2\ep\\
&\le 2K+2\ep,
\end{align*}
which contradicts \ref{rho k} as $\rho(k)\to\infty$.

\end{proof}

\begin{remark}\label{motakis}
Even though the above argument is quantitative, that is, $k$ can be chosen in advance, this does not prove that $(\N^k, d_{\K})_k$ do not equi-coarsely embed into $X$. It is important in the above proof that $\phi$ is defined on entire $c_0$ so that for every $A\subseteq \N$, the maps $\phi^A$ exist. Note that given a family of equi-coarse embeddings $\phi_k:\N^k\to X$ there is no sensible way to define $\phi^A_k$'s as in the above proof. In fact, by an observation due to P. Motakis (see \cite{LPP}), the generalized James space $\J(e_i)$ over the unit vector basis $(e_i)$ of the dual Tsirelson space $T^*$ is quasi-reflexive of order 1 and thus $c_0$ does not coarsely embed into $\J(e_i)$, however, $(\N^k, d_{\K})_k$ equi-coarsely embed into it. In the next two sections we will completely clarify non-equi-coarse embeddings of $(\N^k, d_{\K})_k$ vs non-embedding of $c_0$ into the generalized James and the James tree spaces.
\end{remark}

\section{Non-embedding of the Kalton graphs into the James space}\label{section-J}
In this section, we prove 

\begin{theorem}\label{J}
The Kalton's interlacing graphs $(\N, d_{\K})$ do not equi-coarsely embed into the James space $\J$.
\end{theorem}

As mentioned in the introduction, this theorem was first proved in \cite{LPP}. Our proof is simpler and more importantly it generalizes to give the same result for the James tree spaces. However, this generalization is highly nontrivial and will be given only later in final part of the paper.

Let  $(e_i)$ be the boundedly complete basis for the James space $\J$. Recall that the norm of the James space $\J$ with respect to the boundedly
complete basis (the summing basis) $(e_i)$ is given by 

$$\Big\|\sum_i a_i e_i\Big\|=\sup\left\{\Big(\sum_{j=1}^k\big(\sum_{i\in
I_j}a_i\big)^2\Big)^{1/2}: I_1<\ldots<I_k\right\}$$ where the sup is taken over
all intervals $I_j$'s with no gaps in between.

Let $S$ be the summing functional, that is, $S(\sum_i a_i e_i)=\sum_i a_i$. Then
$S$ is bounded and we may assume $\|S\|=1$. First, we recall a well known fact.

\begin{lemma}\label{J-blocks} Suppose that $(u_i)$ is a block basis of $(e_i)$
in $J$ with $S(u_i)=0$. Then for all $n$, we have
$$\Big(\sum_{i=1}^n \|u_i\|^2\Big)^{1/2}\le\Big\|\sum_{i=1}^n u_i\Big\|\le
2\Big(\sum_{i=1}^n \|u_i\|^2\Big)^{1/2}.$$
\end{lemma}

The left hand side inequality is immediate from the definition and holds for all
block vectors. The right hand side inequality follows from the fact that if $I_1,
\ldots, I_k$ are successive intervals norming $\sum_{i=1}^n u_i$ then if $I_j$'s
contain the support of an $u_i$ entirely, then since $S(u_i)=0$, breaking up the intervals will affect the norm by at most a factor of 2. The inequality then follows from the triangle inequality.

Suppose that we have equi-coarse embeddings $$\phi_k:(\N^k,
d_{\K})\to \J,\ k\in\N.$$ We may assume there are a constant $K$ and non-decreasing
function $\rho$ with $\lim_{t\to\infty}\rho(t)=\infty$ so that for all $\vec{n},
\vec{m}\in \N^k$ and $k\in\N$ we have 

\begin{equation}\label{J-equi-coarse}\rho(d_{\K}(\vec{n}, \vec{m}))\le
\|\phi_k(\vec{n})-\phi_k(\vec{m})\|\le K d_{\K}(\vec{n}, \vec{m}).
\end{equation}

By Theorem \ref{linearization} for almost all tuples $\vec{n}$, $\phi(\vec{n})$ approximately can be written as sum of blocks vectors. Below we will suppress the approximations for simplicity and we will assume all tuples $\vec{n}$ are as in Theorem \ref{linearization}, and of the form 
$$\phi_k(\vec n)=h_0+\sum_{i=1}^k h(n_1,\ldots, n_i)$$ where $h_0$ is a fixed block vector (independent of $\vec n$), and blocks $h(n_1,\ldots, n_i)$ have successive support with respect to $(e_i)$ with $\|h(n_1,\ldots, n_i)\|\le K$ for all $1\le i\le k$.

The key to the argument is Proposition \ref{l_2-blocks} below which roughly says that for almost all interlacing tuples $n_1<m_1<\ldots<n_k<m_k$, the norm of 
$$\phi_k(\vec{n})-\phi_k(\vec{m})=\sum_{i=1}^k h(n_1,\ldots, n_i)-h(m_1,\ldots, m_i)$$
is ($2+\ep$)-equivalent to the $\ell_2$-norm of the sequence of the blocks of differences $(h(n_1,\ldots, n_i)-h(m_1,\ldots, m_i))_{i=1}^k$. In James space 
$\J$, by Lemma \ref{J-blocks} above, this is true for all blocks $(u_i)$ with $S(u_i)=0$ where $S$ is the summing functional. The proof will exploit this fact.

\begin{prop}\label{l_2-blocks}
Fix $M\in\N$ and $\phi_M:\N^M\to \J$ satisfying (\ref{J-equi-coarse}) as above. Let $\sU$ be a Ramsey ultrafilter and $\ep>0$. Then there exists $H=H(\ep,M)\in\sU$ and $(a_i)_{i=1}^M\in \R^M$ such that for all $\vec{n}\in H^M$, $\phi_M(\vec n)=h_0+\sum_{i=1}^M h(n_1,\ldots, n_i)$ satisfy the following.

(i) $\left|\|h(n_1,\ldots, n_i)\|-a_i \right|<\ep/M.$

(ii) $\|(a_i)_{i=1}^M\|_2\le K+\ep.$ 

(iii) For all $\vec n, \vec m\in H^M$ and intervals $I\subseteq [1, M]$ if the restrictions $(n_i)_{i\in I}, (m_i)_{i\in I}$ are interlacing, then, letting $u_i=h(n_1,\ldots, n_i)-h(m_1,\ldots, m_i)$, $i\in I$, we have 
$$\|(a_i)_{i\in I}\|_2\le \Big\|\sum_{i\in I} u_i\Big\|\le
(4+\ep)\|(a_i)_{i\in I}\|_2.$$
\end{prop}

\begin{proof}

First we note that for almost all $\vec{n}$ we have

\begin{equation}\label{uniform-l_2-bound}
\sum_{i=1}^M \|h(n_1, \ldots,
n_i)\|^2\le K^2.
\end{equation}

Indeed, by Theorem \ref{linearization} for $\sU$-large set of interlacing tuples
$m_1<n_1\ldots<m_M<n_M$ we have $\phi(n_1,\ldots, n_M)-\phi(m_1, \ldots, m_M)$
is of the form $\sum_{i=1}^M h(n_1, \ldots, n_i)-h(m_1, \ldots, m_i)$ for some
blocks $h(m_1)<h(n_1)<h(m_1, m_2)<h(n_1, n_2)<\ldots<h(m_1, \ldots, m_M)<h(n_1,
\ldots, n_M)$. Thus, by Lemma \ref{J-blocks} and (\ref{J-equi-coarse}), we have 

\begin{align*}
\sum_{i=1}^M\|h(n_1, \ldots, n_i)\|^2&\le \sum_{i=1}^M\|h(n_1, \ldots, n_i)-h(m_1, \ldots, m_i)\|^2\\
&\le \Big\|\sum_{i=1}^M h(n_1, \ldots, n_i)-h(m_1, \ldots, m_i)\Big\|^2\\
&=\|\phi(n_1, \ldots, n_M)-\phi(m_1, \ldots, m_M)\|^2\le K^2.
\end{align*}

Note that the first inequality above follows from the basis constant being 1.

Let $\ep>0$. By Ramsey ultrafilter we can find a homogeneous set $H=H(\ep, M)\in \sU$ so that
for all $(n_1, \ldots, n_M), (m_1, \ldots, m_M)\in H^M$ with
$n_1<m_1<\ldots<n_M<m_M$ we have 

\begin{align*}
\max_{1\le i\le M}\Big|\|h(n_1, \ldots, n_i)\|-\|h(m_1, \ldots, m_i)\|\Big|<\ep/M,\ \text{and}\\
\max_{1\le i\le M}\Big|S(h(n_1, \ldots, n_i))-S(h(m_1, \ldots, m_i))\Big|<\ep/M.
\end{align*}

Fix $(n_1, \ldots, n_M)\in H^M$ and put $a_i=\|h(n_1, \ldots, n_i)\|$ for $1\le
i\le M$. Then (i), (ii) and (iii) of Proposition \ref{l_2-blocks} now follow from Lemma \ref{J-blocks} and 
\ref{uniform-l_2-bound} (take $u_i=(h(n_1,\ldots, n_i)-h(m_1, \ldots, m_i))$ for
$i\in I$) and standard approximations.

\end{proof}

We will also make use of the following pigeonhole lemma.
  
\begin{lemma}[Pigeonhole]\label{pigeonhole} Let $\ep>0$, $k\in \N$. Let
  $M>k^2K^2/\ep^2$ be a multiple of $k$. If $\sum_{i=1}^{M}a^2_i\le K^2$ then there is
  $1\le N<N+k\le M$ such that
  $$\sum_{i=N}^{N+k}|a_i|<\ep.$$
  \end{lemma}

\begin{proof}
 Put $x_j=\sum_{i=1}^{j}a^2_i$. Then we have $0\le x_1\le \ldots<x_M\le K^2$.
 Let $\ep'=\ep^2/k$. By the pigeonhole principle applied to the sequence
 $x_k, x_{2k}, \ldots, x_{M}$ we must have $l$ such that
 $|x_{lk}-x_{(l+1)k}|<\ep'$. Put $N=lk$. Then by Cauchy-Schwarz
$$\sum_{i=N}^{N+k}|a_i|\le
\Big(\sum_{i=N}^{N+k}|a_i|^2\Big)^{1/2}\sqrt{k}<\ep.$$
\end{proof}

We now return to the proof of Theorem \ref{J}. Let $k\in\N$ be so that
$\rho(k)\ge 10K$, $\ep=1/k$, and $M=M(k,K,\ep)$ be as in Lemma \ref{pigeonhole}.
By (i) of Proposition (\ref{l_2-blocks}) there exist $H\in\sU$ and $(a_i)_{i=1}^M$ such that for all
$(n_1, \ldots, n_M)$ we have, ignoring tiny approximations,
$a_i=\|h(n_1, \ldots, n_i)\|$ for $1\le i\le M$. Let $N$ be as in Lemma \ref{pigeonhole} so that
$\sum_{i=N}^{N+k}a_i<\ep$. Since $\sU$ is Ramsey we may stabilize the $N$ in Lemma \ref{pigeonhole} for all tuples from $H$. Consider two tuples $\vec{n}, \vec{m}$ in $H^M$ of
the form
\begin{align*}
n_1=m_1< \ldots<& n_{N-1}=m_{N-1} \\ 
<n_N< \ldots &< n_{N+k}<m_N<\ldots<m_{N+k}\\
&< n_{N+k+1}=m_{N+k+1}< \ldots< n_M=m_M
\end{align*} 

That is, two tuples are identical except on the interval $[N, N+k]$ where one
comes after the other. Thus $d_{\K}(\vec{n}, \vec{m})=k$. Then
$\phi(\vec{n})-\phi(\vec{m})$ is of the form 
$$u+\sum_{i=N+k+1}^M h(n_1, \ldots,
n_i)-h(m_1, \ldots, m_i)$$ where $u=\sum_{i=N}^{N+k} h(n_1, \ldots, n_i)-h(m_1,
\ldots, m_i)$, and since the first $N-1$ blocks are
identical they cancel out. Then by (ii) and (iii) of Proposition \ref{l_2-blocks} and Lemma \ref{pigeonhole} we have

\begin{align*}
&10K<\rho(k)\le \|\phi(\vec{n})-\phi(\vec{m})\|\\
&\le \left\|\sum_{i=N}^{N+k}h(n_1, \ldots, n_i)- h(m_1, \ldots, m_i)\right\|+\left\|\sum_{i=N+k+1}^M h(n_1, \ldots, n_i)-h(m_1, \ldots, m_i)\right\|\\
&\le 2 \sum_{i=N}^{N+k} a_i +(4+\ep)\left(\sum_{i=N+k+1}^M a_i^2\right)^
{1/2}\\
&\le 2\ep+(4+\ep)(K+\ep),
\end{align*}
which is a contradiction for small $\ep>0$.

\section{Coarse non-universality of dual James spaces}\label{section-gen-james}

Let $(e_i)$ be a basis for a Banach space $E$. We may assume without loss of generality that $(e_i)$ is 1-suppression unconditional as $\J(e_i)$ is naturally isometric to $\J(f_i)$ where $(f_i)$ is `unconditionalization' of $(e_i)$ (see \cite{BHO}).
 The James space $\J(e_i)$ over $(e_i)$ is defined as follows. For $(a_i)\in c_{00}$,   
$$\Big\|\sum a_i u_i \Big\|_{J(e_i)}=\sup\Big\| \sum_{i=1}^k\Big(\sum_{j=p(i)}^{q(i)}a_j\Big)e_{p(i)}\Big\|$$
where the sup is over all $k\in\N$ and $1\le p(1)\le q(1)<p(2)\le q(2)<\ldots<p(k)\le q(k)$. We recall the basic facts about these spaces from \cite{BHO}. 

i) The basis $(u_i)$ is boundedly complete if and only if $c_0$ doesn't linearly embed into $\J(e_i)$.

ii) ${\J(e_i)}^*=[S\cup (u^*_i)_{i=1}^\infty]$ where $S$ is the summing functional (which is bounded on $\J(e_i)$).

iii) If $c_0$ and $\ell_1$ do not linearly embed into $\J(e_i)$ then $\J(e_i)$ is quasi-reflexive of order one.

\begin{corollary}\label{gen-james-graphs}
Let $(e_i)$ be an unconditional basis for a Banach space $E$, and $\J(e_i)$ be the James space over $(e_i)$. 

i) Suppose $\J(e_i)$ doesn't contain a linear copy of $\ell_1$. Then $c_0$ coarsely embeds into $\J(e_i)$ if and only if $c_0$ linearly embeds into $\J(e_i)$.

ii) The Kalton graphs $(\N^k, d_{\K})$  do not equi-coarsely embed into $\J(e_i)$ if  $\ell^n_{\infty}$'s do not belong to the asymptotic structure of $E$. In particular, the Kalton graphs $(\N^k, d_{\K})$  do not equi-coarsely embed into $\J_p=\J(e_i)$ where $(e_i)$ is the unit vector basis of $\ell_p$ for $1<p<\infty$.
\end{corollary}

\begin{proof}
i) Suppose $c_0$ does not linearly embed into $\J(e_i)$. Then $\J(e_i)$ is quasi-reflexive of order one, and the result follows from Theorem \ref{quasireflexive}.

ii) This follows from a more general Theorem \ref{tree-space} proven in next section.

\end{proof}

We will show next that $\ell_1$ assumption in (i) of the above Corollary is not necessary. 

\begin{theorem}\label{J(e_i)-c_0}
Let $(e_i)$ be an unconditional basis for a Banach space $E$, and $\J(e_i)$ be the James space over $(e_i)$. Then $c_0$ coarsely embeds into $\J(e_i)$ if and only if $c_0$ linearly embeds into $\J(e_i)$.
\end{theorem}

\begin{proof}
Suppose $c_0$ does not linearly embed into $\J(e_i)$. Then by \cite{BHO} the basis $(u_i)$ of $\J(e_i)$ is boundedly complete. Suppose $\phi
:c_0\to \J(e_i)$ is a coarse embedding. The proof is a slight variation of the proof of Theorem \ref{spreading} so we will only briefly indicate the required additional argument which additionally exploits the fact that the summing functional $S$ is bounded on $\J(e_i)$ and the block sequences $(w_i)$ with $S(w_i)=0$ are unconditional (See Proposition 2.1 of \cite{BHO}). Suppose we have the same set up as in the proof of Theorem \ref{spreading} up to equation (\ref{A-B-interlacing}). By (\ref{same sums}) the summing functional is essentially zero on the blocks of differences
\begin{align*}h^A(m_1)-h^B(m_1)<h^A(n_1)-h^B(n_1)<\ldots<h^A(m_1,\ldots,
m_k)-\\h^B(m_1,\ldots, m_k)<h^A(n_1, \ldots, n_k)-h^B(n_1, \ldots, n_k)
\end{align*} and therefore by Proposition 2.1 of \cite{BHO}, this block sequence is $2$-suppression unconditional (for small $\ep>0$ in (\ref{same sums})).

Thus
\begin{align*}
\rho(k)&\le\|\phi^A(\vec{n})-\phi^B(\vec{n})\|\le 2\|\phi^A(\vec{n})-\phi^A(\vec{m})+\phi^B(\vec{m})-\phi^B(\vec{n})\|\\
&\le 2\|\phi^A(\vec{n})-\phi^A(\vec{m})\|+2\|\phi^B(\vec{m})-\phi^B(\vec{n})\|\\
&\le 4K
\end{align*}
which is a contradiction for large $k$.

\end{proof}

\section{Non-embedding of the Kalton graphs into generalized James tree spaces}\label{Sec:JamesTree}

Recall that the James tree space $\jt$ is the space of real valued functions on the
binary tree $T=2^{<\omega}$ with norm
$$\|x\|=\sup\left(\sum_{j=1}^k \Big|\sum_{t\in S_j}x(t)\Big|^2 \right)^{1/2}$$
where the sup is taken over all sets of {\em disjoint segments} $(S_j)_{j=1}^k$. A segment $S$
is a finite interval of a branch in $T$. Note that the subspaces of functions
restricted to a single branch is isomorphic to the James space $\J$. Informally, the
James tree space is obtained by `hanging' $\J$ on every branch of the binary
tree. The node basis $(u_\alpha)_{\alpha\in T}$, when ordered in a natural way that is
compatible with the tree order, is a boundedly complete basis. The dual $\jt^*$
is non-separable: For every branch $b$, the  functional $S_b$ (summing over $b$) is bounded and has norm one, and for two distinct branches $b, b'$, $\|S_b-S_{b'}\|_{\jt^*}\ge \sqrt 2$.

The {\em generalized James tree spaces} are obtained by replacing the $\ell_2$-norm in the above by other norms. Let $(e_i)$ be a normalized basis for some Banach space $E$.  As in the original James tree space, in which case $(e_i)$ is the unit vector basis of $\ell_2$, the James tree space over $(e_i)$,  $\jt(e_i)$, is defined on the linear space of all finitely supported functions $x:T\to \R$ where  $T=2^{<\omega}$ is the full binary tree. As before $S=[\alpha, \beta]$ denotes segments which are interval subsets of branches in $T$, and $S(x)=\sum_{\gamma\in S}x(\gamma)$. We fix an ordering $o:T\to\N$ compatible with the tree order. If $S=[\alpha, \beta]$ then put $o(S)=o(\alpha)$. Then the norm on $\jt(e_i)$ is given by

$$\|x\|=\sup\Big\{\big\|\sum_{i=1}^k S_i(x)e_{o(S_i)}\big\|_{E}: k\in \N\ \text{and}\ (S_i)_{i=1}^k\ \text{are disjoint segments in}\ T\Big\}.$$
As with the generalized James spaces, we may assume without loss of generality that $(e_i)$ is 1-suppression unconditional as $\jt(e_i)$ is naturally isometric to $\jt(f_i)$ where $(f_i)$ is `unconditionalization' of $(e_i)$ (see \cite{BHO}).

The vectors $E(x):=\sum_{i=1}^k S_i(x)e_{o(S_i)}$ in $E$ are called the {\em representatives} of $x$.
The node basis $u_{\alpha}(\beta)=\delta_{\alpha,\beta}$ is a monotone basis in the ordering $o(\alpha)$. Segments $S$ are norm one linear functionals on $\jt(e_i)$. Similarly, the branch functionals $f_b(x)=\sum_{\gamma\in b}x(\gamma)$ where $b$ is an infinite branch are also norm one functionals.

 Bellenot, Odell, and Haydon \cite{BHO} proved that if we start with a space $E$ with a boundedly complete basis $(e_i)$ then the basis $(u_{\alpha})$ of $\jt(e_i)$ is boundedly complete as well. In this section, we will explore the following:

\begin{Question}
Let $(e_i)$ be a boundedly complete basis for some Banach space $E$. When do the Kalton interlacing graphs $(\N^k, d_{\K})$ equi-coarsely  embed into the James tree space $\jt(e_i)$?
\end{Question}

As noted in \cite{LPP}, if $(e_i)$ is the unit vector basis of dual Tsirelson space $T^*$ (which is boundedly complete), then the Jamesification $\J(e_i)$ has the summing basis $(s_i)$ of $c_0$ as a spreading model generated by the basis. Thus the family $(\N^k, d_{\K})$ equi-coarsely embed (in fact, equi-Lipschitz embed) into $\J(e_i)$. In particular, the same holds for the James tree space $\jt(e_i)$. We will show that this example is essentially the only exception. If $c_0$ is not (asymptotically) finitely block representable in $[e_i]$, that is, $\ell^n_{\infty}$'s do not belong to the asymptotic structure in the sense of \cite{MMT}, then the family $(\N^k, d_{\K})$ do not equi-coarsely embed into the James tree space $\jt(e_i)$.

\begin{theorem}\label{tree-space}
Let $E$ be a Banach space with an unconditional basis $(e_i)$ and let $\jt(e_i)$ be the James tree space over $(e_i)$. Suppose $\ell^n_{\infty}$'s do not belong to the asymptotic structure of $E$. Then the Kalton's interlacing graphs $(\N^k, d_{\K})$ do not equi-coarsely embed into $\jt(e_i)$.
\end{theorem}

In particular, we have

\begin{corollary}\label{JT}
The Kalton's interlacing graphs $(\N^k, d_{\K})$ do not equi-coarsely embed into the James tree space $\jt$.
\end{corollary}

The proof is a non-trivial generalization of the argument given for the James space in Section \ref{section-J}. The non-trivial part of the generalization is to reduce the embeddings into a finite set of branches, which is given in the subsection below. The reduction is done in a more general setting assuming only that the node basis of $\jt(e_i)$ is boundedly complete.

Suppose that we have equi-coarse embeddings $$\phi_k:(\N^k,
d_{\K})\to \jt(e_i),\ \ k\in\N.$$ We may assume there are a constant $K$ and non-decreasing
function $\rho$ with $\lim_{t\to\infty}\rho(t)=\infty$ so that for all $\vec{n},
\vec{m}\in \N^k$ and $k\in\N$ we have 

\begin{equation}\label{rho}
\rho(d_{\K}(\vec{n}, \vec{m}))\le
\|\phi_k(\vec{n})-\phi_k(\vec{m})\|\le K d_{\K}(\vec{n}, \vec{m}).
\end{equation}

We fix a large $k$ (to be determined at the end of the proof) and drop subscript from $\phi_k$ and write $\phi$ for brevity. As in previous sections, we make some simplifying assumptions on $\phi$. Given a Ramsey ultrafilter $\sU$, by Theorem \ref{linearization} and Remark \ref{homog-tree} there exists $H\in\sU$ so that for all $\vec{n}\in H^k$, $\phi(\vec{n})$ can be written, ignoring tiny approximations, as a sum of blocks vectors

\begin{equation}\label{form}\phi(\vec n)=h_0+\sum_{i=1}^k h(n_1,\ldots, n_i)\end{equation} where blocks $h(n_1,\ldots, n_i)$ have successive support with respect to the node basis of $\jt(e_i)$. We may assume that the support of $h_0$ is contained in $[0,0^+]$, and the supports of $h(n_1,\ldots, n_i)$ are contained in $[n^-_i, n^+_i]$ where the intervals are {\em intervals of levels of the tree} (rather than basis intervals), and $n^-_i$ is the immediate predecessor of $n_i$ and $n^+_i$ is the immediate successor of $n_i$ in $H$, and $0^+$ is the first element of $H$ (Lemma \ref{se} below).

\subsection{Reduction of embeddings to a finite set of branches}\label{section-reduction}


If $\sS$ is a set of branch segments (which are always taken to be pairwise
node disjoint), then  we let $\|\phi(\vec n)\|_{\sS}$
denote the approximation to $\| \phi(\vec n)\|$ computed using the branch segments in $\sS$, that is, 
$$\|\phi(\vec n)\|_{\sS}=\Big\|\sum_{S\in \sS}S(\phi(\vec n))e_{o(S)}\Big\|_E.$$



More generally, if $F$ is a set of nodes in the tree, we let $\|\phi(\vec n)\|_{F}$
denote the norm computed using sets of branch segments $\sS$ which {\em respect $F$}, that is, 
each $S \in \sS$ is a subset of $F$. Thus
$$\|\phi(\vec n)\|_{F}=\sup_{\sS}\Big\|\sum_{S\in \sS}S(\phi(\vec n))e_{o(S)}\Big\|_E$$
where the sup is over all disjoint collections $\sS$ that respect $F$.


The following is immediate from Theorem \ref{linearization} as mentioned above.

\begin{lemma} \label{se}
Let $\eta>0$. 
There is an infinite $H\in \sU$ such that for all $\vec n \in H^k$ there exists $\sS=\sS(\vec n,\eta)$ whose segments start and end at nodes 
of length in $[0,0^+]\cup \bigcup_i [n_i^-,n_i^+]$ (where $n_i^-$ and $n_i^+$
refer to the set $H$) such that $\|\phi(\vec n)\|_{\sS} > \| \phi(\vec n)\|-\eta$. 
\end{lemma}

We henceforth assume that the function $\sS(\vec n,\eta)$ and the set $H$ have the property
as stated in Lemma~\ref{se}. That is, we may assume 
that all segments $S \in \sS(\vec n,\eta)$ start and end in one of the $[n_i^-,n_i^+]$
(along with $[0,0^+]$). 

\begin{definition} Let $\sU$ be a Ramsey ultrafilter. Let $\rho_0$ be the supremum of all real numbers $r$ such that there is an $H\in \sU$
and a function $\sF_0$ such that $\sF_0(\vec n)$ is a finite set of disjoint infinite branch segments  
all of which begin at a node of length in $[0,0^+]$ such that for all $\vec n\in H^k$,
$$\|\phi(\vec n)\|_{\sF_0(\vec n)} \geq r.$$ 


We say a segment $s$ {\em diverges} from an infinite branch segment $b$ at level $n$
if the backward extensions of $s$ and $b$ diverge in the tree $T$ at a node of height $n$.

\end{definition}

For $\eta>0$, we let $\sF_0^\eta$ be a function such that there is a homogeneous $H\in\sU$ 
such that for all $\vec n \in H^k$, $\sF_0^\eta(\vec n)$ is a disjoint set of infinite branch segments 
starting below $0^+$ (defined relative to $H$), with $\| \varphi(\vec n)\|_{\sF_0^\eta(\vec n)} >\rho_0-\eta$. 
When the $\eta$ is fixed and there is no danger of confusion, we will write $\sF_0(\vec n)$
for $\sF_0^\eta(\vec n)$.

\begin{lemma} \label{ml1}
Let $\eta>0$. Then there is a homogeneous $H\in\sU$ such that for all $\vec n\in H^k$
we have $\| \phi(\vec n) \|_C< 2 \eta$, where $C$ is the union of the nodes which diverge 
from $\sF_0(\vec n)$ below $n_1^-$.
\end{lemma}

\begin{proof}
Consider the partition of $k+2$ tuples $a<b<n_1<\cdots < n_k$ according to whether 
$\| \phi(\vec n)\|_A <\eta$, where $A$ is the set of nodes in $T$ which diverge 
from $\sF^\eta_0(\vec n)$ between $a$ and $b$. By the boundedly completeness of the norm, on the 
homogeneous side the stated property holds. Indeed, suppose to the contrary that there is a homogeneous set $H\in\sU$ on which the property fails. Then consider $a_1<b_1<\ldots<a_t<b_t<n_1<\ldots<n_k$ in $H$.  For each $a_j<b_j<n_1<\ldots<n_k$ for $1\le j\le t$, we have $\| \phi(\vec n)\|_{A_j} \ge\eta$, where $A_j$ is the set of nodes in $T$ which diverge 
from $\sF^\eta_0(\vec n)$ between $a_j$ and $b_j$. Since $A_j$'s consist of pairwise disjoint sets of branches, it follows that $\| \phi(\vec n)\|\ge \big\|\sum_{i=1}^t\sum_{S\in A_i}S(\phi(\vec n))e_{o(S)}\big\|_E$, which tends to infinity as $t$ gets larger by the boundedly completeness of the norm $\|\cdot\|_E$. Thus, this is contradiction for a large enough $t$ since $\| \phi(\vec n)\|$ is bounded by $kK$ by (\ref{rho}). 
We will use this type of argument often and will refer to it as a {\em boundedly completeness argument}.

Therefore, the stated property holds. Fix $H\in\sU$ homogeneous for the partition. 
Let $0^+$ denote the least element of $H$. Then by homogeneity, for almost all $\vec n$ we have 
$\| \phi(\vec n)\|_A <\eta$, where $A$ is the collection of nodes which diverge from 
$\sF^\eta_0(\vec n)$ between $0^+$ and $n_1^-$. On the other hand, by the definition of $\sF^\eta_0$
we have that for almost all $\vec n$ that $\| \phi(\vec n)\|_B<\eta$, where $B$ 
is the set of nodes which diverge from $\sF^\eta_0(\vec n)$ below $0^+$. Thus, for almost all $\vec n$,
if $C=A\cup B$ then $\| \phi(\vec n)\|_C <2\eta$. 
\end{proof}

Fix now a small $\eta$, and let $H$ be a homogeneous set as in Lemma~\ref{ml1}. 
Thus, for almost all $\vec n$ we have that  $\| \phi(\vec n) \|_C< 2 \eta$, 
where $C$ is the union of the nodes which diverge from $\sF_0(\vec n)$ below $n_1^-$.
Recall $\sF_0(\vec n)$ is a finite set of infinite branch segments, all of which start 
below a fixed level $0^+$ of $T$. By the finite additivity of the ultrafilter, we may assume that
the size of $\sF_0(\vec n)$ does not depend on $\vec n$. We may order these branches 
lexicographically, and enumerate them as $\sF_0(\vec n)= b^0_0(\vec n),\dots, b^0_{p_0}(\vec n)$. 
These branches depend on $\vec n$, but we will sometimes just call them $b^0_0,\dots, b^0_{p_0}$. 
Note that if $s_i= b^0_i(\vec n)\res 0^+$, then for almost all $\vec n$ the $s_i$ do not depend on 
$\vec n$, and the $s_0,\dots, s_{p_0}$ are distinct (that is, $b^0_i$ and $b^0_j$ 
for $i \neq j$ split before $0^+$ by the node disjointness of the $b^0_i$'s). 
The sequences $s_0,\dots, s_{p_0}$ are henceforth fixed.

\begin{lemma} \label{ml2}
For each $0 \leq i \leq p_0$ there is an infinite branch segment $b^*_i$ starting below $0^+$
such that for almost all $\vec n$ we have that $b^0_i(\vec n)\res n_1^-= b_i^* \res n_1^-$.
\end{lemma}

\begin{proof}
Consider the partition of $2k+1$ tuples $a <n_1<\cdots <n_k< n'_1<\cdots 
< n'_k$ according to whether $b^0_i(\vec n) \res a= b^0_i(\vec n') \res a$. Suppose that on the homogeneous 
side the stated property does not hold, and let $H\in\sU$ be homogeneous for the contrary side. 
Fix $a_0 \in H$, and consider $l$-tuples $\vec n^0< \vec n^1<\dots <\vec n^l$ coming from $H^k$, all above $a_0$. 
Since $a_0$ is fixed, for a large enough $l$ (by pigeonhole) there must be $c<d$ such that $b^0_i(\vec n^c)\res a_0= b^0_i(\vec n^d)\res a_0$. 
But then $\{ a_0\} \cup \vec n^c \cup \vec n^d$ violates the homogeneity of $H$.
So, on the homogeneous side we have that if $\vec n< \vec n'$ then 
$b^0_i(\vec n)\res n_1^-= b^0_i(\vec n') \res n_1^-$. It then follows easily that for any two 
tuples $\vec n$, $\vec n'$ from $H$ with $n_1= \min \{ n_i,n'_i\}$ that 
$b^0_i(\vec n)\res n_1^-= b^0_i(\vec n') \res n_1^-$. We let $b^*_i =\bigcup_{\vec n} b^0_i(\vec n) \res n_1^-$
which is then well-defined. 

\end{proof}

Let $\sF^*=\{ b^*_i \colon 0\leq i \leq p_0\}$. We note that $\sF^*$ is a fixed (independent of $\vec n)$
set of infinite branch segments of size $p_0$, with each $b^*_i$ extending $s_i$. In particular, 
all of the pairs of distinct branches from $\sF^*$ split below a fixed level $0^+$ of $T$.

\begin{lemma} \label{ml3}
For almost all $\vec n$ we have $\| \phi(\vec n)\|_{A} <\eta$
where $A$ is the set of nodes which diverge from $\sF^*$ outside of levels between $n_i^-$ and $n_i^+$ for some $i$.
More precisely, there is a homogeneous set $H$ (which then 
defines the notions $n_i^-$, $n_i^+$, etc.) such that for almost all $\vec n$
the stated inequality holds.
\end{lemma}

\begin{proof}
Consider the partition of $3k$ tuples $a_1<n_1<b_1<a_2<n_2<b_2<\cdots <a_k<n_k<b_k$
according to whether $\| \phi(\vec n) \|_A <\eta$ where $A$ is the collection of nodes 
which diverge from $\sF^*$ at a level not in $\bigcup_{i=1}^k [a_i,b_i]$. 
Suppose that on the homogeneous side the stated property fails, and fix such an $H$. 
Then there is an $i$ such that on the homogeneous side $\| \phi(\vec n)\|_{A_i} >\frac{\eta}{k}$
where $A_i$ is the set of nodes which diverge from $\sF^*$ at a level in $[b_{i-1},a_i]$ (put $b_0=0$). 
We then consider $\vec n\in H^k$ where in between $n_{i-1}<n_i$ there are $t$ many pairs $(b_{i}^1, a_{i+1}^1) <\cdots <( b_{i}^t, a_{i+1}^t)$ where each $b_i^j$ is less than 
$a_{i+1}^j$. By homogeneity, $\| \phi(\vec n)\|_{B^j_i} > \frac{\eta}{k}$ for each $j\in [1,t]$
where $B^j_i$ is the set of nodes that diverge from $\sF^*$ at levels in $[b^j_i, a^j_{i+1}]$. 
Since $B^j_i$ is node disjoint from $B^{j'}_i$ for $j\neq j'$, a boundedly completeness argument gives that 
$\| \phi(\vec n)\|$ gets arbitrarily large for a large $t$, which is a contradiction.

So, on the homogeneous side the stated property holds. Let $H$ be homogeneous for the 
partition, which defines the notions $n_i^-$, $n_i^+$, etc.\ The homogeneity of $H$
then gives the statement of the lemma. 

\end{proof}

\begin{lemma} \label{ml4}
There are functions $\sF_1(\vec n),\dots, \sF_k(\vec n)$ such that each $\sF_i(\vec n)$
is a finite set of pairwise node disjoint infinite branch segments all of which diverge 
from $\sF^*$ at levels between $n_i^-$ and $n_i^+$, the union of the nodes in 
$\sF^*\cup \bigcup_{1\leq i\leq k} \sF_i(\vec n)$ is a tree, and such that for almost all $\vec n$
we have $\| \phi(\vec n)\|_A <\eta$, where $A$ is the union of the nodes not in this tree. 
\end{lemma}

\begin{proof}
Fix $1\leq i\leq k$ and we define $\sF_i(\vec n)$. To define $\sF_i(\vec n)$ we simply 
let $\sF'_i(\vec n)$ be a finite set of infinite branch segments all of which diverge from $\sF^*$
between $n_i^-$ and $n_i^+$ and such that $\| \phi(\vec n)\|_B< \frac{\eta}{k}$
where $B$ is the set of nodes which diverge from $\sF^*$ between $n_i^-$ and $n_i^+$
which are not in $\cup \sF'_i$. We can then easily get infinite branch segments $\sF_i$ 
such that $\sF^* \cup \sF'_i(\vec n)$ and $\sF^* \cup \sF_i(\vec n)$ have the same set of nodes
and $\sF_i(\vec n)$ are pairwise node disjoint and branch from $\sF^*$ in $(n_i^-,n_i^+)$. 
The $\sF_i(\vec n)$ satisfy the statement of the lemma.
\end{proof}

We now consider interlaced tuples $(\vec n,\vec m)= n_1<m_1<\cdots < n_k<m_k$. 
For such a tuple we have the infinite branch segments $\sF_i(\vec n)$ defined for $1 \leq i \leq k$
and also the infinite branch segments $\sF_i(\vec m)$. Note that any two distinct infinite branch 
segments $b_1$ and $b_2$ from  $\bigcup_i \sF_i(\vec n)$ and $ \bigcup_i \sF_i(\vec m)$ respectively 
are node disjoint, as they diverge from $\sF^*$ at different levels.

The following lemma is immediate from Lemma~\ref{ml4}. 

\begin{lemma} \label{ml5}
Let $\eta>0$. For almost all $(\vec n,\vec m)$ we have $\| \phi(\vec n)\|_A, \| \phi(\vec m)\|_B<\eta$
where $A$ is the collection of nodes not in $\sF^*\cup \bigcup_{1\leq i\leq k} \sF_i(\vec n)$
and likewise $B$ is the collection of nodes not in 
$\sF^*\cup \bigcup_{1\leq i\leq k} \sF_i(\vec m)$.
\end{lemma}


\begin{lemma} \label{ia1}
Let $\phi \colon \N^k \to \jt(e_i)$ be such that if the $k$-tuples $\vec n$, $\vec m$ 
are interlaced (i.e., $n_1<m_1<\cdots<n_k<m_k$) then $\| \phi(\vec n)-\phi(\vec m)\|<K$,
for some constant $K$. 
Then there is a finite set $\sF^*$ of infinite branch segments in $T$ and a set $H\in \sU$ 
such that for all $\vec n\in H^k$ we have $\| \phi(\vec n)\|_A< (K+1)$
where $A=T\sm \cup \sF^*$ is the set of nodes in $T$ not in any branch segment of $\sF^*$. 

\end{lemma}

\begin{proof}
Let $0<\eta<1/2$. Let $H\in\sU$ and $\sF^*$, $\sF_i$ for $1 \leq i \leq k$ (so $\sF^*$ is a fixed set, and the $\sF_i$ are functions of $\vec n$) be as in above lemmas. 
Consider an interlaced tuple $n_1<m_1<\cdots <n_k<m_k$ from $H$. Let $A$ be the nodes of $T$ not in 
$\cup \sF^*$. Then $A=A_1\cup A_2$ where $A_1 = T\sm \cup (\sF^*\cup \bigcup_{1\leq i \leq k} \sF_i(\vec n))$
and $A_2= \cup (\sF^*\cup \bigcup_{1\leq i \leq k} \sF_i(\vec n)) \sm \cup \sF^*$. 
From Lemma~\ref{ml5} we have that $\| \phi(\vec n)\|_{A_1}<\eta$ and $\| \phi(\vec m)\|_{A_2}<\eta$. Thus

\begin{align*}
\|\phi(\vec n)\|_A&\le \|\phi(\vec n)\|_{A_2}+ \|\phi(\vec n)\|_{A_1}
\le \|\phi(\vec n)\|_{A_2}+\eta\\
& \leq \|\phi(\vec n)-\phi(\vec m)\|_{A_2}+ \| \phi(\vec m)\|_{A_2}+ \eta\\
& \leq \|\phi(\vec n)-\phi(\vec m)\|_{A_2}+2 \eta\\
&\le \|\phi(\vec n)-\phi(\vec m)\|+2\eta\\
&\le K+2\eta.
\end{align*}

\end{proof}

We summarize the contents of Lemmas~\ref{ml1}--\ref{ia1} in the following. Note that we merely assume that $\jt(e_i)$ has boundedly complete node basis.

\begin{theorem} \label{mls}
Let $\jt(e_i)$ be a James tree space with a boundedly complete node basis. Suppose $\phi \colon \N^k\to \jt(e_i)$ satisfies (\ref{rho}).  Let $\sU$ be a Ramsey ultrafilter on $\N$.
Then for any $\eta>0$ there is a $H\in \sU$, a finite set of pairwise disjoint infinite 
branch segments $\sF^*$ in $T$, and functions $\sF_1,\dots, \sF_k$ with domain $\N^k$
satisfying:
\begin{enumerate}
\item
For all $\vec n\in H^k$, $\sF_i(\vec n)$ is a finite set of infinite branch segments, 
each of which diverges from $\sF^*$
at a level between $n_i^-$ and $n_i^+$, where $n_i^-$, $n_i^+$ are defined with respect to 
$H$ (e.g., $n_i^+$ is the least element of $H$ greater than $n_i$).
Also, the collection of nodes in $T$ in the branch segments $\sF^*\cup \bigcup_{1\leq i \leq k} \sF_i(\vec n)$
forms a tree. 
\item
For all $\vec n \in H^k$, $\| \phi(\vec n)\|_A <\eta$, where $A$ is the collection of nodes 
not in $\sF^*\cup \bigcup_{1\leq i \leq k} \sF_i(\vec n)$.

\item
For all $\vec n\in H^k$ we have $\| \phi(\vec n)\|_A< (K+1)$
where $A=T\sm \cup \sF^*$ is the set of nodes in $T$ not in any branch of $\sF^*$. 
\end{enumerate}
\end{theorem}

\subsection{Pigeonhole lemma and stabilization on a finite set of branches}
The second part of the proof involves a pigeonhole trick and a stabilization of embedding $\phi$ on the finite set of branches $\sF^*$ from Theorem \ref{mls}. The assumption that $\ell_{\infty}^n$'s do not belong to the asymptotic structure of $E$ in Theorem \ref{tree-space} is essential in the following simple pigeonhole lemma which generalizes Lemma \ref{pigeonhole}.

Recall that $\{X\}_n$ denotes the asymptotic structure of $X$, see the discussion before Corollary \ref{lin-asy} for details. 

\begin{lemma}\label{M-k} Let $(e_i)$ be a 1-unconditional basis for $E$. Suppose that $\ell^n_{\infty}$'s do not belong to the asymptotic structure of $E$. Let $K\ge 1$. For all $k\in\N$ and $\ep>0$ there exists $M=M(k, K, \ep)$ such that if $(u_i)_{i=1}^M\in \{\jt(e_i)\}_M$ and $\|\sum_{i=1}^{M}u_i\|\le K$ then there exists an interval $1\le N<N+k<M$ such that $\sum_{i=N}^{N+k}\|u_i\|<\ep$.
\end{lemma}

\begin{proof}
Otherwise, there exist $k$ and $\ep>0$ such that for all $M$, $\|\sum_{j=1}^{M/k}\sum_{i=(j-1)k}^{jk}u_i)\|\le K$ with $\|w_j\|\ge \ep$ where $w_j=\sum_{i=(j-1)k}^{jk}u_i$. Let $E(w_j)$ be norming representatives of $w_j$'s in $E$. This implies $(E(w_j))_{j=1}^{M/k}$ is $K/\ep$ equivalent to the unit vector basis of $\ell^{\infty}_{M/k}$. Since $M$ is arbitrary, and $(u_i)_{i=1}^M\in \{\jt(e_i)\}_M$ implies the supports of $E(w_j)$'s are arbitrarily far out with respect to $(e_i)$, and this means that we have $\ell^{\infty}_{n}\in\{E\}_n$ for all $n$ (with constant $K/\ep$), contradicting the assumption.
\end{proof}

The proposition below is a generalization of Proposition \ref {l_2-blocks}.

First, some notation. Suppose $u_1<u_2<\ldots<u_M$ is a block sequence where the support of $u_i$ is contained in between the levels $n^-_i$ and $n^+_i$ and $n^+_i<n^-_{i+1}$, that is, there is at least one level gap between the supports of $u_i$'s, and let $u=\sum_{i=1}^Mu_i$. Among the representatives of $u$ in $E$ consider those $\sum_{j=1}^k t_j(u)e_{o(t_j)}$ given by {\em trivial segments} $t_j$'s. We say that a finite segment $t$ is {\em trivial} if $t\subseteq [n^-_i, n^+_i]$ (recall that $[n^-_i, n^+_i]$ refers to  nodes at levels between $n^-_i$ and $n^+_i$) for some $i$. Of course, this notion depends on the given blocking $u=\sum_{i=1}^M u_i$.  Thus a trivial representative is of the form
$$E^t(\phi(\vec n)):=\sum_{i=0}^M \sum_{t^i_j\in A_i}t^i_j(h(n_1, \ldots, n_i))e_{o(t^i_j)}$$ where $A_i$ is a set of trivial segments $t^i_j$ contained in $[n^-_i, n^+_i]$. By $P_{\sF^*}$ we denote the natural projection (restriction) on the finite set of infinite branch segments $\sF^*$.

\begin{prop}\label{trivial-segments}
Let $\sU$ be a Ramsey ultrafilter and $\ep>0$.
Let $\jt(e_i)$ be as in Theorem \ref{tree-space}, for $M\in \N$ let $\phi:(\N^M, d_{\K})\to \jt(e_i)$ be as in (\ref{rho}), and $\sF^*$ be as in Theorem \ref{mls}.  Then there exists $H\in\sU$ such that for all tuples $\vec{n}$ and $\vec{m}$ in $H^M$, if on an interval $I\subseteq [1,M]$, $(n_i)_{i\in I}$ and $(m_i)_{i\in I}$ are interlacing then 
\begin{align*} \Big\|\sum_{i\in I}P_{\sF^*}h(n_1, \ldots n_i)-P_{\sF^*}h(m_1, \ldots m_i)\Big\| \le
4K+2\ep.
\end{align*}
\end{prop}

\begin{proof}
For almost all $\vec n$ we have $$\phi(\vec n)=h_0+\sum_{i=1}^M h(n_1,\ldots, n_i)$$ where blocks $h(n_1,\ldots, n_i)$ are supported in $[n^-_i,n^+_i]$. First, we note that for almost all $\vec n$ the trivial representatives of $\sum_{i=1}^M h(n_1,\ldots, n_i)$ satisfy $$\Big\|E^t(\sum_{i=1}^M h(n_1,\ldots, n_i))\Big\|_E\le K.$$
Indeed, consider $n_1<m_1<\ldots<n_M<m_M$ and let $A_i$ be a set of trivial segments $t^i_j$ contained in $[n^-_i, n^+_i]$, then 
\begin{align*}
\Big\|E^t(\sum_{i=1}^M h(n_1,\ldots, n_i))\Big\|_E&=\Big\|\sum_{i=1}^M \sum_{t^i_j\in A_i}t^i_j(h(n_1, \ldots, n_i))e_{o(t^i_j)}\Big\|_E\\
&=\Big\|\sum_{i=1}^M \sum_{t^i_j\in A_i}t^i_j(h(n_1, \ldots, n_i)-h(m_1,\ldots,m_i))e_{o(t^i_j)}\Big\|_E\\
&\le \|\phi(\vec n)-\phi(\vec m)\|\\
&\le K.
\end{align*}

Let $\sF^*=\{b_1,\ldots, b_p\}$ be the finite set of disjoint infinite branch segments. By Ramsey there exists a homogeneous set $H\in\sU$ such that for all $\vec n, \vec m\in H^M$ we have 
$$\Big|S_{b_j}\big(h(n_1, \ldots, n_i)-h(m_1, \ldots, m_i)\big)\Big|<\ep/M,\ 1\le j\le p, 1\le i\le M.$$ 

Fix $\vec n, \vec m\in H^M$ and let $I$ be an interval in $[1,M]$ as in the statement of the theorem, that is, $(n_i)_{i\in I}$ and $(m_i)_{i\in I}$ are interlacing. Put $u_i=P_{\sF^*}h(n_1, \ldots n_i)-P_{\sF^*}h(m_1, \ldots m_i)$, $i\in I$ for brevity and consider a norming representative for $\sum_{i\in I}u_i$. That is,
$$\Big\|\sum_{i\in I}u_i\Big\|\le \Big\|\sum_{j=1}^l s_j\big(\sum_{i\in I}u_i\big)e_{o(s_j)}\Big\|_E+\ep,$$
for some set of segments $(s_j)_{j=1}^l$. 

Since $S_{b_j}(u_i)$'s are essentially zero, if $s_j$'s
contain the support of an $u_i$ entirely, then  breaking up the intervals will affect the norm by at most a factor of 2.
Therefore, we may assume each $s_j$ contained in one $u_i$. Put $s_j=s_j^{i(1)}\cup s_j^{i(2)}$ where $s_j^{i(1)}$ and $s_j^{i(2)}$ are the segments of $s_j$ which intersects only the support of $h(n_1,\ldots, n_i)$ and of $h(m_1,\ldots, m_i)$, respectively, and put $s_j=s_j^{i(1)}=s_j^{i(2)}$ if $s_j$ is already a trivial segment. Thus the above quantity is less than or equal to
\begin{align*}
2\Big\|\sum_{i\in I}\sum_{j=1}^l s_j^{i(1)}(h(n_1,\ldots,n_i))e_{o(s_j)}\Big\|_E +2\Big\|\sum_{i\in I}\sum_{j=1}^l s_j^{i(2)}(h(m_1,\ldots, m_i))e_{o(s_j)}\Big\|_E+2\ep,
\end{align*}
which is less than or equal to $4K+2\ep$ by the first part of the proof.

\end{proof}

\subsection{Proof of Theorem \ref{tree-space}}

We are now ready to proceed with the proof of Theorem \ref{tree-space}.
Let $k\in\N$ be so that
$\rho(k)\ge 10K$, $\ep>0$, and $H\in\sU$ be as in Proposition \ref{trivial-segments}.

Let $M=M(k,K,\ep)$ be as in Lemma \ref{M-k}.  Using Lemma \ref{M-k} and $\sU$ is a Ramsey ultrafilter we may assume the $N$ in Lemma \ref{M-k} is stabilized for all tuples in $H^M$.
Consider $\vec {n}, \vec{m}$ in $H^M$ so that

\begin{align*}
n_1=m_1< \ldots<& n_{N-1}=m_{N-1} \\ 
<n_N< \ldots &< n_{N+k}<m_N<\ldots<m_{N+k}\\
&< n_{N+k+1}=m_{N+k+1}< \ldots< n_M=m_M
\end{align*} 

That is, two tuples are identical except on the interval $[N, N+k]$ where one
comes after the other. Thus $d_{\K}(\vec{n}, \vec{m})=k$. Then
$\phi(\vec{n})-\phi(\vec{m})$ is of the form 
$$\sum_{i=N}^{N+k} h(n_1, \ldots, n_i)-\sum_{i=N}^{N+k}h(m_1,
\ldots, m_i)+\sum_{i=N+k+1}^M h(n_1, \ldots,
n_i)-h(m_1, \ldots, m_i),$$  and since the first $N-1$ blocks are
identical they cancel out. Then by Theorem \ref{mls}, Proposition \ref{trivial-segments} and Lemma \ref{M-k} we have

\begin{align*}
&10K<\rho(k)\le \|P_{\sF^*}\big[\phi(\vec{n})-\phi(\vec{m})\big]\|+\|P_{(\sF^*)^c}\big[\phi(\vec{n})-\phi(\vec{m})\big]\|\\
&\le \left\|\sum_{i=N}^{N+k}P_{\sF^*}h(n_1, \ldots, n_i)- \sum_{i=N}^{N+k}P_{\sF^*}h(m_1, \ldots, m_i)\right\|\\
&\ \ +\left\|\sum_{i=N+k+1}^M P_{\sF^*}h(n_1, \ldots, n_i)-P_{\sF^*}h(m_1, \ldots, m_i)\right\|+2K+2\\
&\le 2\ep+2K+2\ep+2K+2,
\end{align*}
which is a contradiction for small $\ep>0$.

\section{Further questions}

The general question which motivates all our previous work is still open: Can $c_0$ coarsely embed into a separable dual Banach space? As a consequence of our work in Section \ref{Sec:JamesTree}, however, the following also presents itself.

\begin{Question}\label{gen-JT}
Let $\jt(e_i)$ be a James tree space with boundedly complete node basis $(u_\alpha)_{\alpha\in T}$. Does $c_0$ coarsely embed into $\jt(e_i)$? As a concrete example, consider $(e_i)$ to be the unit vector basis of the dual Tsirelson space $T^*$. 

\end{Question}

\end{document}